\documentclass[11pt]{article}
\usepackage{mathrsfs}
\usepackage{amsfonts}
\usepackage{amssymb}
\usepackage{amsthm}
\usepackage{color}
\usepackage{mathrsfs, amssymb}
\usepackage{subfigure}
\usepackage{float}
\usepackage{comment}
\usepackage{graphicx}
\usepackage{psfrag}
\usepackage{epstopdf}
\usepackage{amsmath}

\allowdisplaybreaks[4]

\newcommand \dis {\displaystyle}

\newcommand{\dy}{\mathrm{d} y}

\newcommand{\ds}{\mathrm{d} s}

\newcommand{\dl}{\mathrm{d} l}
\renewcommand{\d}{\mathrm{d}}
\newcommand{\ee}{\mathrm{e}}

\newcommand{\R}{\mathbb{R}}

\numberwithin{equation}{section}
\newtheorem{theorem}{Theorem}[section]

\newtheorem{assumption}[theorem]{Assumption}

\newtheorem{claim}[theorem]{Claim}

\newtheorem{definition}[theorem]{Definition}

\newtheorem{lemma}[theorem]{Lemma}

\newtheorem{proposition}[theorem]{Proposition}
\newtheorem{remark}[theorem]{Remark}

\makeatletter

\newcommand{\Rmnum}[1]{\expandafter\@slowromancap\romannumeral #1@}
\makeatother

\begin{document}

\title {\bf Existence and Stability of Random Transition Waves for Nonautonomous Fisher-KPP Equations with Nonlocal Diffusion}
\author{Min Zhao
\thanks{Corresponding author. Email addresses: minzhao@mail.bnu.edu.cn (M. Zhao), ryuan@bnu.edu.cn (Rong Yuan). https://orcid.org/0000-0001-7722-318X (M. Zhao)}
, Rong Yuan}
\date{\it Laboratory of Mathematics and Complex Systems (Ministry of Education), \\School of Mathematical Sciences, Beijing Normal University, \\Beijing 100875, People's Republic of China}
\maketitle

\vspace*{-1cm}
\begin{center}

\begin{minipage}[t]{13cm}

{\bf Abstract:}\quad  In this paper, we study the existence and stability of random transition waves for time
heterogeneous Fisher-KPP Equations with nonlocal diffusion. More specifically, we consider general
time heterogeneities both for the nonlocal diffusion kernel and the reaction term. We use  the comparison principle of the scalar equation and the method of upper and lower solutions to investigate the existence of random transition wave solution when the wave speed is large enough. In addition, we show the stability of random transition fronts for non-autonomous Fisher-KPP equations with nonlocal diffusion.

\end{minipage}
\end{center}

\section{Introduction}

In this paper, we study the existence and stability of  transition wave solution for the following non-autonomous Fisher-KPP equations with nonlocal dispersal
\begin{equation}\label{1.1}
\frac{\partial u}{\partial t}(t, x)=\int_{\R}J(\theta_{t}\omega, y)[u(t,x-y)-u(t,x)]\dy+a(\theta_{t}\omega)u(t,x)(1-u(t,x)),
\end{equation}
where $t\in\R$, $x\in\R$, $\omega\in\Omega,$ $(\Omega,\mathcal{F},\mathbb{P},\{\theta_{t}\}_{t\in\R})$ is an ergodic metric dynamical system on $\Omega$, and $a:\Omega\to(0,\infty)$ is measurable, and $a^{\omega}(t):=a(\theta_{t}\omega)$ is locally H\"{o}lder continuous for every $\omega\in\Omega$. Here $J^{\omega}:=J(\theta_{t}\omega,y)$ denotes a nonnegative dispersal kernal. The state variable $u = u(t, x)$ represents the population density of the species located at time $t$ and the spatial position $x\in\R$. 

System \eqref{1.1} describes the temporal and spatial evolution of species invasions into some empty environments.
Here we not only give the time-dependent Fisher-KPP term, but also consider the time-varying dispersal kernel function $J^{\omega}$. The kernel function $J^{\omega}$ satisfies the following assumptions.
\begin{assumption}[Kernel $J^{\omega}=J(\theta_{t}\omega, y)$]\label{Ass1.2}
The kernel $J^{\omega}:\Omega \times \mathbb{R} \rightarrow[0, \infty)$ satisfies the following assumptions:
\begin{description}
\item[(i)] The function $J^{\omega}$ is measurable, nonnegative and $J(\cdot, y) \in L_{+}^{\infty}(\Omega,\mathcal{F},\mathbb{P})$ for almost every $y \in \mathbb{R}$;
\item[(ii)] The map $\widetilde{J}^{\omega}: y \mapsto J(\cdot, y)$ from $\mathbb{R}$ into $L^{\infty}(\Omega,\mathcal{F},\mathbb{P})$ is measurable and integrable, namely $\widetilde{J}^{\omega} \in L^{1}\left(\mathbb{R} ; L^{\infty}(\Omega,\mathcal{F},\mathbb{P})\right)$;
\item[(iii)] Its abscissa of convergence 
$$
\sigma(\widetilde{J}^{\omega})=\sup \left\{\mu \geq 0 \text { : the improper integral } \int_{-\infty}^{\infty} \mathrm{e}^{\mu y}\widetilde{J}^{\omega}(y) \mathrm{d}y \text { converge in } \R\right\}
$$
satisfies 
$$
\sigma(\widetilde{J}^{\omega})>0.
$$
\end{description}
\end{assumption}
At present, many scholars are devoted to studying the existence and stability of traveling wave solutions, see \cite{Aronson75,Aronson78,Sattinger76,Uchiyama78}. Fisher \cite{Fisher37} and Kolmogorov, Petrovsky and Piskunov \cite{Kolmogorov37} did pioneering work on the traveling wave solution and take-over properties of the Fisher-KPP reaction-diffusion equation.
Note that for periodic media in time and/or space, it is natural to extend the classical traveling wave solution to the so-called periodic traveling wave solutions or pulsating traveling fronts. For general time and/or space dependent environment,  \cite{Berestycki07,Berestycki12} extend traveling wave solutions in the classical sense to the so-called transition fronts or generalized traveling waves. 

For the time and/or space heterogeneous reaction-diffusion equation, there are many studies on existence of traveling wave solutions and the spreading speed, see \cite{Berestycki05,Freidlin79,Hamel11,Hamel16,Liang06,Liang07,Liang10}. Shen first proved the existence of transition waves for bistable nonlinearities \cite{Shen06}, and further studied the monostable equation with time unique ergodic coefficients \cite{Shen11}. Shen also studied the existence of generalized traveling waves in time recurrent and space periodic monostable equations \cite{Shen11a}. Salako and Shen studied the existence of transition fronts of random and nonautonomous Fisher-KPP equations \cite{Salako19}, respectively. Nadin and Rossi \cite{Nadin12} investigated transition wave solution of the monostable equations in general time heterogeneous environment. They furthermore explored the existence of transition waves for Fisher-KPP equations with general time-heterogeneous and space-periodic coefficients \cite{Nadin15}. Nadin \cite{Nadin09} studied the existence of traveling fronts in space-time periodic media.
Rossi and Ryzhik studied the transition waves for a class of space-time dependent monostable equations \cite{Rossi14}.
Ambrosio, Ducrot and Ruan \cite{Ambrosio20} studied the generalized traveling wave solutions of a non-cooperative reaction-diffusion system in a general time heterogeneous environment. 

In addition to the random reaction-diffusion equation, nonlocal dispersal is more reasonable for some species to travel for some distance, and their movement and interactions may occur between non-adjacent spatial locations \cite{Bao19,Jin,Dong}. 
There are many studies devoted to the existence, nonexistence, and stability problems of traveling wave solutions to the nonlocal diffusion equation \cite{Bates97,Carr04,Coville05,Coville07,Lutscher05}.

For the time and/or space heterogeneous nonlocal diffusion equation, Jin and Zhao \cite{Jin} studied the spatial dynamics of a periodic population model with dispersal. Coville et al. \cite{Coville13} studied the 
pulsating fronts for nonlocal dispersion and KPP nonlinearity.
Bao et al. \cite{Bao16} explored the traveling wave solutions of Lotka-Volterra competition systems with nonlocal dispersal in periodic habitats. Shen et al. \cite{Shen15} studied the transition fronts in nonlocal Fisher-KPP equations in time heterogeneous media. Shen also studied the stability of transition waves and positive entire solutions of Fisher-KPP equations with time and space dependence \cite{Shen17}.
Dcurot and Jin \cite{Ducrot21} studied the generalized travelling fronts for non‑autonomous Fisher‑KPP
equations with nonlocal diffusion. They dealt with a general time heterogeneous both for dispersal kernel function and KPP nonlinearity. 

For the random reaction-diffusion equation, Shen \cite{Shen04} first introduced a notion of traveling
waves in general random media, which is a natural extension of the classical notion of traveling waves. More specifically, She explored the traveling waves in diffusive random media, including time and/or space recurrent, almost periodic, quasiperiodic, periodic ones as special cases. Mierczy${\rm\acute{n}}$ski et al. studied the uniform persistence for 
random parabolic Kolmogorov systems \cite{Mierczynski04}.
Shen et al. \cite{S-Shen17} studied front propagation phenomena of ignition type reaction-diffusion equations in random media. However, there are a few studies on the random nonlocal diffusion equation. In this paper, we focus on the existence and stability of random traveling wave solutions of nonlocal diffusion Fisher-KPP equations. Specifically, different from Ducrot and Jin's work \cite{Ducrot21}, we further consider the stability of random transition waves. For solving the stability problem of system \eqref{1.1}, the construction of the lower solution plays an important role. Based on this, we also consider the existence of traveling wave solutions by constructing lower solutions different from those in the Ducrot and Jin's research \cite{Ducrot21}.

Now we are ready to introduce the definition of a random transition front for \eqref{1.1}. Shen gave a similar definition for random transition front, see \cite{Shen04}.
\begin{definition}\label{def-wave}
	An entire solution $u(t,x;\omega)$ is called a random traveling wave solution or a random transition front of \eqref{1.1} connection $1$ and $0$ if for a.e. $\omega\in\Omega$, 
	\begin{equation}\label{1.2}
		u(t,x;\omega)=U(x-C	(t;\omega),\theta_{t}\omega),
	\end{equation}
	where $U(x,\omega)$ and $C(t;\omega)$ are measurable in $\omega$, and for a.e. $\omega\in\Omega$,
	\begin{equation}\label{1.3}
		0<U(x,\omega)<1
	\end{equation}
	and
	\begin{equation}\label{1.4}
		\lim_{x\rightarrow-\infty}U(x,\theta_{t}\omega)=1,\quad
		\lim_{x\rightarrow+\infty}U(x,\theta_{t}\omega)=0 \text{ uniformly in }t\in\R.
	\end{equation}
	Moreover, if $U_{x}(x, \omega)<0$ for a.e. $\omega\in\Omega$ and all $x\in\R$, $u(t,x;\omega)$ is said to be a nonotone random transition front.
\end{definition}
We observe that $u(t, x;\omega)=v(t, x-C(t ; \omega))$ with $C(t ; \omega)$ being differential in $t$ solves \eqref{1.1} and satisfies $c(t ; \omega)=C^{\prime}(t ; \omega)$ if and only if $v(t, x)$ satisfies
\begin{equation}\label{3.1b}
	\partial_{t}v(t, x)=c(t; \omega)\partial_{x}v(t, x)+\int_{\R} J(\theta_{t}\omega,y)[v(t, x-y)-v(t, x)]\dy+a(\theta_{t}\omega)v(t, x)(1-v(t, x)).
\end{equation}

This paper is organized as follows. In the next section, we establish some preliminary results and state our main results.
In Section 3, we mainly consider the existence of random transition waves of system \eqref{1.1}. We get the results by constructing appropriate upper and lower solutions and using the comparison principle of the scalar equation. In Section 4,  we prove the asymptotic stability of random transition waves for non-autonomous Fisher-KPP equations with nonlocal diffusion.

\section{Preliminaries and main results}
In this section, we first recall the definitions and some properties of least mean and upper mean values for functions in $L^{\infty}(\R)$. To do so we need to introduce some notations that will be used throughout this paper. Second, we present our main results on the existence and stability of random transition fronts for system \eqref{1.1}.

First, we introduce the space aspects.
Let
$$
C_{\text {unif}}^{b}(\mathbb{R})= \{u \in C(\R) | u \text{ is bounded and uniformly continuous }\}
$$
with norm
$\Vert u\Vert_{\infty}=\sup_{x\in\R}|u(x)|$ for $u\in C_{\text {unif}}^{b}(\mathbb{R})$.

For given $u_{0} \in X:=C_{\text {unif }}^{b}(\mathbb{R})$ and $\omega \in \Omega$, let $u\left(t, x ; u_{0}, \omega\right)$ be the solution of \eqref{1.1} with $u\left(0, x ; u_{0}, \omega\right)=u_{0}(x)$. 
Note also that $u \equiv 0$ and $u \equiv 1$ are two constant solutions of \eqref{1.1}.

Now we introduce the following notations
and assumption related to \eqref{1.1}. Let 
\begin{equation}\label{1.2a}
	\lfloor a\rfloor(\omega)=\liminf_{t-s \rightarrow+\infty}\dfrac{1}{t-s}\int_{s}^{t}a(\theta_{\tau}\omega)\d\tau:=\lim_{r\rightarrow+\infty}\inf_{t-s\geq r}\dfrac{1}{t-s}\int_{s}^{t}a(\theta_{\tau}\omega)\d\tau,
\end{equation}
and
\begin{equation}\label{1.3a}
	\lceil a\rceil(\omega)=\limsup_{t-s \rightarrow+\infty}\dfrac{1}{t-s}\int_{s}^{t}a(\theta_{\tau}\omega)\d\tau:=\lim_{r\rightarrow+\infty}\sup_{t-s\geq r}\dfrac{1}{t-s}\int_{s}^{t}a(\theta_{\tau}\omega)\d\tau.
\end{equation}
Notice that
\begin{equation}\label{1.4a}
	\lfloor a\rfloor(\theta_{t}\omega)=\lfloor a\rfloor(\omega) \text{ and } \lceil a\rceil(\theta_{t}\omega)=\lceil a\rceil(\omega), \forall t\in\R,
\end{equation}
and that
\begin{equation}\label{1.5a}
	\lfloor a\rfloor(\omega)=\liminf_{t,s\in\mathbb{Q},t-s \rightarrow+\infty}\dfrac{1}{t-s}\int_{s}^{t}a(\theta_{\tau}\omega)\d\tau \text{ and }\lceil a\rceil(\omega)=\limsup_{t,s\in\mathbb{Q},t-s \rightarrow+\infty}\dfrac{1}{t-s}\int_{s}^{t}a(\theta_{\tau}\omega)\d\tau.
\end{equation}
Then by the countability of the set $\mathbb{Q}$ of rational numbers, both $\lfloor a\rfloor(\omega)$ and $\lceil a\rceil(\omega)$ are measurable in $\omega$.

Next, we introduce the following two assumptions:
\begin{description}
	\item[(H1)] $0<\inf_{t\in\R} a^{\omega}(t)\leq\sup_{t\in\R} a^{\omega}(t)<\infty, \text{ for a.e. }\omega\in\Omega.$
\end{description}
\begin{description}
	\item[(H2)] $0 < \lfloor a\rfloor(\omega)\leq \lceil a\rceil(\omega)<\infty, \text{ for a.e. }\omega\in\Omega.$
\end{description}
\begin{remark}
\begin{description}
	\item[(i)] We observe from (H1) that 
	$$
	0 < \inf_{t\in\R} a^{\omega}(t)\leq \lfloor a\rfloor(\omega)\leq \lceil a\rceil(\omega)\leq\sup_{t\in\R} a^{\omega}(t)<\infty, \text{ for a.e. }\omega\in\Omega.
	$$
\end{description}
\begin{description}
	\item[(ii)] Here we specifically state that the assumption (H1) is mainly used to prove the existence of random traveling wave solution, see Theorem \ref{Th2.3}, while the weaker hypothesis (H2) is used to prove the stability of random transition front, see Theorem \ref{Th2.4}.
\end{description}
\end{remark}
Note that assumption (H2) implies that $\lfloor a\rfloor(\cdot), \lceil a\rceil(\cdot) \in L^{1}(\Omega, \mathcal{F}, \mathbb{P})$ (see Lemma 2.1 in \cite{Salako19}). 
By assumption (H2) and the ergodicity of the metric dynamical system $\left(\Omega, \mathcal{F}, \mathbb{P},\left\{\theta_{t}\right\}_{t \in \mathbb{R}}\right)$, there are $\lfloor a\rfloor, \lceil a\rceil \in \mathbb{R}^{+}$and a measurable subset $\Omega_{0} \subset \Omega$ with $\mathbb{P}\left(\Omega_{0}\right)=1$ such that
\begin{equation}\label{1.6}
\dis\left\{\begin{array}{l}
\theta_{t} \Omega_{0}=\Omega_{0}, \quad \forall t \in \mathbb{R}, \\
\dis\liminf _{t-s \rightarrow \infty} \frac{1}{t-s} \int_{s}^{t} a\left(\theta_{\tau} \omega\right) d \tau=\lfloor a\rfloor, \quad \forall \omega \in \Omega_{0}, \\
\dis\limsup _{t-s \rightarrow \infty} \frac{1}{t-s} \int_{s}^{t} a\left(\theta_{\tau}\omega\right) d \tau=\lceil a\rceil, \quad \forall \omega \in \Omega_{0}.
\end{array}\right.
\end{equation}
Throughout this paper, $\lfloor a\rfloor$ and $\lceil a \rceil$ are referred to as the least mean and the upper mean of $a(\cdot)$, respectively. 

The following lemma gives reformulation of the least and upper
mean value, see \cite{Nadin12,Salako19} for more details.
\begin{lemma}\label{Lem2.1}
	Suppose that $a(\theta_{t}\omega)\in L^{\infty}(\Omega)$, $t\in\R$ and that $0<\lfloor a\rfloor \leq \lceil a\rceil<\infty$, where
	$$
	\lfloor a\rfloor=\liminf _{t -s \rightarrow \infty} \frac{1}{t-s} \int_{s}^{t} a(\theta_{\tau}\omega) d \tau, \quad \lceil a\rceil=\limsup _{t -s \rightarrow \infty} \frac{1}{t-s} \int_{s}^{t} a(\theta_{\tau}\omega) d \tau .
	$$
	Then
	$$
	\lfloor a\rfloor=\sup _{A_{\omega} \in W_{\text {loc }}^{1, \infty}(\R) \cap L^{\infty}(\R)} \operatorname{essinf}_{\tau \in \mathbb{R}}\left(a(\theta_{\tau}\omega)-A^{\prime}_{\omega}(\tau)\right), \text{ for a.e.  }\omega\in\Omega,
	$$
	and
	$$
	\lceil a \rceil=\inf _{A_{\omega} \in W_{\text {loc }}^{1, \infty}(\R) \cap L^{\infty}(\R)} \operatorname{esssup}_{\tau \in \mathbb{R}}\left(a(\theta_{\tau}\omega)-A^{\prime}_{\omega}(\tau)\right), \text{ for a.e.  }\omega\in\Omega.
	$$
\end{lemma}
\begin{proof}
	The proof of this lemma follows from a proper modification of the proof of \cite[Lemma 2.2]{Salako19}. For the sake of completeness we give a proof here. We fix $0<\alpha<\lfloor a\rfloor$. It follows from  $\lceil a \rceil<\infty$ that there exists $T>0$ such that
	\begin{equation}\label{2.6d}
		\alpha<\frac{1}{T} \int_{s}^{s+T} a(\theta_{\tau}\omega)d\tau<2 \lceil a \rceil, \quad \forall s \in \mathbb{R}, \omega\in\Omega.
	\end{equation}
	Define
	\begin{equation*}
		A_{\omega}(t) =\int_{k T}^{t}\left(a(\theta_{\tau}\omega)-\alpha_{k}\right) d \tau, \quad \forall t \in[k T,(k+1) T],
	\end{equation*}
where
\begin{equation*}
	\alpha_{k} :=\frac{1}{T} \int_{k T}^{(k+1) T} a(\theta_{\tau}\omega) d\tau, \quad \forall k \in \mathbb{Z},\omega\in\Omega.
\end{equation*}
It is clear that $A_{\omega} \in W_{\text {loc }}^{1, \infty}(\mathbb{R}) \cap L^{\infty}(\mathbb{R})$ with
	\begin{equation}\label{2.7d}
		\alpha_{k}=a(\theta_{t}\omega)-A^{\prime}_{\omega}(t), \text { for } t \in(k T,(k+1) T),\omega\in\Omega.
	\end{equation}
On the one hand, by using \eqref{2.6d}, we deduce that
$$
\|A_{\omega}\|_{\infty} \leq 2 T\lceil a \rceil \text{ and }\alpha<\alpha_{k}, \forall k \in \mathbb{Z}, \omega\in\Omega.
$$
It follows from \eqref{2.7d} that
$$
\alpha \leq \sup _{A_{\omega} \in W_{\text {loc }}^{1, \infty}(\mathbb{R}) \cap L^{\infty}(\mathbb{R})} \operatorname{essinf}_{t \in \mathbb{R}}\left(a(\theta_{t}\omega)-A^{\prime}_{\omega}(t)\right), \omega\in\Omega.
$$
Since $\alpha$ is arbitrarily chosen less than $\lfloor a\rfloor$, we derive that
	$$
	\lfloor a\rfloor \leq \sup _{A_{\omega} \in W_{\mathrm{loc}}^{1, \infty}(\mathbb{R}) \cap L^{\infty}(\mathbb{R})} \operatorname{essinf}_{t \in \mathbb{R}}\left(a(\theta_{t}\omega)-A^{\prime}_{\omega}(t)\right), \omega\in\Omega.
	$$
	On the other hand, for each given $A_{\omega} \in W_{\text {loc }}^{1, \infty}(\mathbb{R}) \cap L^{\infty}(\mathbb{R})$ and $t>s$, we have that for $\omega\in\Omega$
	$$
	\begin{aligned}
		\frac{1}{t-s} \int_{s}^{t} a(\theta_{\tau}\omega)d \tau & \geq \operatorname{essinf}_{\tau \in \mathbb{R}}\left(a(\theta_{\tau}\omega)-A^{\prime}_{\omega}(\tau)\right)+\frac{(A_{\omega}(t)-A_{\omega}(s))}{t-s} \\
		& \geq \operatorname{essinf}_{\tau \in \mathbb{R}}\left(a(\theta_{\tau}\omega)-A^{\prime}_{\omega}(\tau)\right)-\frac{2\|A_{\omega}\|_{\infty}}{t-s}.
	\end{aligned}
	$$
	Hence
	$$
	\lfloor a\rfloor=\liminf _{t-s \rightarrow \infty} \frac{1}{t-s} \int_{s}^{t} a(\theta_{\tau}\omega) d \tau \geq \operatorname{essinf}_{\tau \in \mathbb{R}}\left(a(\theta_{\tau}\omega)-A^{\prime}_{\omega}(\tau)\right), \forall A_{\omega} \in W_{\mathrm{loc}}^{1, \infty}(\mathbb{R}) \cap L^{\infty}(\mathbb{R}),\omega\in\Omega.
	$$
	This completes the proof of the lemma.
\end{proof}

Next, we define the following two functions, which are strongly related to the main results of this work. We consider the function $L:\R\times\Omega\times[0, \sigma(\widetilde{J}^{\omega})) \rightarrow\R$ and $L(t;\cdot,\mu)\in L^{\infty}(\Omega,\mathcal{F},\mathbb{P})$ 
given by
\begin{equation}\label{2.2}
	L(t;\omega,\mu):=\int_{\mathbb{R}} J(\theta_{t}\omega, y) \mathrm{e}^{\mu y} \mathrm{~d} y,\quad \mu \in[0, \sigma(\widetilde{J}^{\omega})), \omega\in\Omega,
\end{equation}
and the function $C:\R\times\Omega\times(0, \sigma(\widetilde{J}^{\omega})) \rightarrow \R$ and 
$C\in L^{\infty}(\Omega,\mathcal{F},\mathbb{P})$ given by
\begin{equation}\label{2.3}
	C(t;\omega,\mu)=\int_{0}^{t}c(s;\omega,\mu)\ds, \quad \mu \in(0, \sigma(\widetilde{J}^{\omega})), \omega \in \Omega,
\end{equation}
where 
\begin{equation}\label{2.3c}
	c(t;\omega,\mu):=\mu^{-1}\left[\int_{\mathbb{R}} J(\theta_{t}\omega, y)\left(\mathrm{e}^{\mu y}-1\right) \mathrm{d} y+a(\theta_{t}\omega)\right], \mu \in(0, \sigma(\widetilde{J}^{\omega})), \omega \in \Omega.
\end{equation}
\begin{proposition}\label{Pro2.2}
	Let Assumption \ref{Ass1.2} be satisfied. Then the following properties hold:
	\begin{description}
		\item[(i)] The maps $L$ defined above in \eqref{2.2} is of class $C^{1}$ from $\R\times\Omega\times(0, \sigma(\widetilde{J}^{\omega}))$ into $\R$.
		\item[(ii)] Consider the sets
		$$
		M=\left\{\mu \in(0, \sigma(\widetilde{J}^{\omega})): \exists\; \mu^{\prime} \in(\mu,\sigma(\widetilde{J}^{\omega})), \forall k \in\left(\mu, \mu^{\prime}\right],\lfloor c(t;\omega,\mu)-c(t;\omega,k)\rfloor>0\right\},
		$$
		and
		$$
		\widetilde{M}=\left\{\mu \in(0, \sigma(\widetilde{J}^{\omega})): \exists\; \mu^{\prime} \in(\mu, \sigma(\widetilde{J}^{\omega})),\left\lfloor c(t;\omega,\mu)-c\left(t;\omega,\mu^{\prime}\right)\right\rfloor>0\right\} .
		$$
		Then one has $M=\widetilde{M}$ and there exists $\mu^{*} \in(0, \sigma(\widetilde{J}^{\omega})]$ such that
		$$
		M=\left(0, \mu^{*}\right).
		$$
		\item[(iii)] One also has:
		$$
		\left\lfloor-\frac{\mathrm{d} c(t;\omega,\mu)}{\mathrm{d} \mu}\right\rfloor>0, \forall \mu \in\left(0, \mu^{*}\right) \text { and }\left\lfloor-\frac{\mathrm{d} c\left(t;\omega,\mu^{*}\right)}{\mathrm{d} \mu}\right\rfloor=0 \text { if } \mu^{*}<\sigma(\widetilde{J}^{\omega}) .
		$$
		\item[(iv)] The function $\mu \mapsto\lfloor c(t;\omega,\mu)\rfloor$ is decreasing on $M$.
	\end{description} 
\end{proposition}
\begin{proof}
	The proof is similar to the Ducrot and Jin \cite{Ducrot21}, so we omit it.
\end{proof}

\begin{proposition}[Comparison principle]\label{Prop2.4}
Let $t_{0} \in \mathbb{R}$ and $T>0$ be given. Let $J^{\omega}:\Omega \times \mathbb{R} \rightarrow[0, \infty)$  be a measurable kernel such that the map $\omega \mapsto \int_{\mathbb{R}} J(\theta_{t}\omega, y) \mathrm{d} y$ is bounded. Let $\underline{u}$ and $\bar{u}$ be two uniformly continuous functions defined from $\left[t_{0}, t_{0}+T\right] \times \mathbb{R}$ into the interval $[0,1]$ such that for each $x \in \mathbb{R}$, the maps $\underline{u}(\cdot, x)$ and $\bar{u}(\cdot, x)$ both belong to $W^{1,1}\left(t_{0}, t_{0}+T\right)$, satisfying $\underline{u}\left(t_{0}, \cdot\right) \leq \bar{u}\left(t_{0}, \cdot\right)$ and, for all $x \in \mathbb{R}$ and for almost every $t \in\left(t_{0}, t_{0}+T\right)$,
$$
\begin{aligned}
\partial_{t} \bar{u}(t, x) & \geq \int_{\mathbb{R}} J(\theta_{t}\omega, y)[\bar{u}(t, x-y)-\bar{u}(t, x)] \mathrm{d} y+a(\theta_{t}\omega)\bar{u}(t,x)(1-\bar{u}(t,x)), \\
\partial_{t} \underline{u}(t, x) & \leq \int_{\mathbb{R}} J(\theta_{t}\omega, y)[\underline{u}(t, x-y)-\underline{u}(t, x)] \mathrm{d} y+a(\theta_{t}\omega)\underline{u}(t,x)(1-\underline{u}(t,x)).
\end{aligned}
$$
Then $\underline{u} \leq \bar{u}$ on $\left[t_{0}, t_{0}+T\right] \times \mathbb{R}$.
\end{proposition}
\begin{proof}
	The proof is similar to the Ducrot and Jin \cite{Ducrot21}, so we omit it.
\end{proof}
\begin{lemma}[\cite{Shen04}]\label{lem2.5d}
Let $\left(\Omega,\left\{\theta_{t}\right\}_{t \in \mathbb{R}}\right)$ be ergodic and $h \in$ $L^{1}(\Omega, \mathcal{F}, \mathbb{P})$ ($h$ is real-valued), then there is $\Omega_{0} \in \mathcal{F}$ with $\mathbb{P}\left(\Omega_{0}\right)=1$ such that
	$$
	\lim _{t \rightarrow \infty} \frac{1}{t} \int_{0}^{t} h\left(\theta_{s} \omega\right) \mathrm{d} s=\int_{\Omega} h(\omega) \mathrm{d} \mathbb{P}(\omega)
	$$
	for all $\omega \in \Omega_{0}$.
\end{lemma} 
Using the above notation, our next result ensures the existence of random transition front for system \eqref{1.1} with the speed function $c(t;\omega,\mu)$, for each $\mu \in\left(0, \mu^{*}\right)$.
\begin{theorem}\label{Th2.3}
	Assume that (H1) holds and for each $\mu \in\left(0, \mu^{*}\right)$,  where $\mu^{*}$ is defined in Proposition \ref{Pro2.2}. 
\item {(1)}
Problem \eqref{1.1} possesses a monotone random transition wave solution $u(t, x;\omega)$
$=U^{\mu}\Big(x-C(t ; \omega, \mu),$$ \theta_{t} \omega\Big)$ with $C(t ; \omega, \mu)=\int_{0}^{t} c(s ; \omega, \mu) d s$, where 
$c(t ; \omega, \mu)$ is defined in \eqref{2.3c}.
Moreover, for any $\omega \in \Omega_{0}$, 
\begin{equation}
	\lim _{x \rightarrow \infty} \sup _{t \in \mathbb{R}}\left|\frac{U^{\mu}\left(x, \theta_{t} \omega\right)}{e^{-\mu x}}-1\right|=0 \text{ and }\lim _{x \rightarrow-\infty} \sup _{t \in \mathbb{R}}\left|U^{\mu}\left(x, \theta_{t} \omega\right)-1\right|=0.
\end{equation}
\item {(2)} There exist $c^*\in\R$ and $U^{\mu}_{*}(\cdot)\in C_{unif}^{b}(\R;\R)$ such that for a.e. $\omega\in\Omega$,
\begin{equation}\label{2.10d}
	\lim_{t\rightarrow\infty}\dfrac{C(t;\omega,\mu)}{t}=c^{*},
\end{equation}
\begin{equation}\label{2.11d}
	\lim_{t\to\infty}\dfrac{1}{t}\int_{0}^{t}U^{\mu}(x,\theta_{s}\omega)\ds=U^{\mu}_{*}(x),\quad\forall x\in\R.
\end{equation}
\end{theorem}
Next, we study the stability of random transition fronts of \eqref{1.1}. 
\begin{theorem}\label{Th2.4}
	Assume that (H2) hold and $ c(t; \omega, \mu)$ is defined in \eqref{2.3c}. Then for given $\mu \in\left(0, \mu^{*}\right)$, the random wave solution $u(t, x)=U\left(x-C(t ; \omega, \mu), \theta_{t} \omega\right)$ with 
	$$
	\lim _{x \rightarrow \infty} \frac{U\left(x ; \theta_{t} \omega\right)}{e^{-\mu x}}=1 \text{ and } C(t ; \omega, \mu)=\int_{0}^{t} c(s ; \omega, \mu)\ds,
	$$
	is asymptotically stable, that is, for any $\omega \in$ $\Omega_{0}$ and $u_{0} \in C_{\mathrm{unif}}^{b}(\mathbb{R})$ satisfying that
	\begin{equation}\label{2.4}
		\inf _{x \leq x_{0}} u_{0}(x)>0, \quad \forall x_{0} \in \mathbb{R}, \quad \lim _{x \rightarrow \infty} \frac{u_{0}(x)}{U(x-C(0 ; \omega, \mu), \omega)}=1,
	\end{equation}
	there holds
	$$
	\lim _{t \rightarrow \infty}\left\|\frac{u\left(t, \cdot ; u_{0}, \omega\right)}{U\left(\cdot-C(t ; \omega, \mu), \theta_{t} \omega\right)}-1\right\|_{\infty}=0.
	$$
\end{theorem}

\section{Existence of Random Transition Fronts} 
In this section, we study the existence of random transition fronts of \eqref{1.1}, see Theorem \ref{Th2.3}. We divide the proof of Theorem \ref{Th2.3} into two steps: (1). we give some lemmas to construct the upper and lower solutions of \eqref{1.1}; (2). we construct the limit behavior to complete the proof of the theorem. 

In this section, we shall always suppose that (H1) holds.
Let $\Omega_{0}$ be as in \eqref{1.6}. Therefore, we have
$$
0<\inf_{t\in\R} a^{\omega}(t)\leq\sup_{t\in\R} a^{\omega}(t)<\infty, \text{ for a.e. }\omega\in\Omega_0.
$$

Recall that $u(t, x;\omega)=v(t, x-C(t ; \omega))$ with $C(t ; \omega)$ being differential in $t$ solves \eqref{1.1} if and only if $v(t, x)$ satisfies \eqref{3.1b}.
Hence, to prove the existence of random transition front of \eqref{1.1} of the form $u(t, x;\omega)=U\left(x-C(t; \omega), \theta_{t} \omega\right)$ for some differentiable $C(t ; \omega)$ and some $U(x, \omega)$ which is measurable in $\omega$ and 
$$
U\left(-\infty, \theta_{t} \omega\right)=1 \text{ and }U\left(\infty ; \theta_{t} \omega\right)=0, \text{ uniformly in }t,
$$
it is equivalent to prove the existence of entire solutions of \eqref{3.1b} (with $c(t ; \omega)=C^{\prime}(t ; \omega)$ ) of the form $v(t, x)=V(t, x ; \omega)$ such that
$$
\left\{\begin{array}{l}
	\omega \mapsto V(t, x ; \omega), \quad \omega \mapsto C(t ; \omega) \text { are measurable, } \\
	V(t, x ; \omega)=V\left(0, x ; \theta_{t} \omega\right), \\
	\lim _{x \rightarrow-\infty} V(t, x ; \omega)=1 \quad \text { and } \lim _{x \rightarrow \infty} V(t, x ; \omega)=0 \text { uniformly in } t .
\end{array}\right.
$$

\subsection{Construction of sub and super solutions}
In this section, we give some lemmas to construct upper and lower solutions of \eqref{3.1b}.

\begin{lemma}\label{Lem3.1}
  Suppose that (H1) holds. Let $\omega \in \Omega_{0}$ and $0<\mu<\mu^{*}$. Let
	$$
	\phi_{+}^{\mu}(x)=\min \left\{1, \phi^{\mu}(x)\right\},
	$$
	where $\phi^{\mu}(x)=\ee^{-\mu x}.$
	Then
	$$
	v\left(t, x ; \phi_{+}^{\mu}(\cdot),\omega\right) \leq \phi_{+}^{\mu}(x) \quad \forall t>0, x \in \mathbb{R},\omega\in\Omega_0.
	$$
\end{lemma}
\begin{proof}
	Since $a(\omega)>0$ for every $\omega \in \Omega_0$ and the function $\phi^{\mu}$ satisfies
	$$
	\phi_{t}^{\mu}(x)=c(t ; \omega, \mu) \phi_{x}^{\mu}(x)+\int_{\R} J(\theta_{t}\omega,y)[\phi^{\mu}(x-y)-\phi^{\mu}(x)]\dy+a(\theta_{t}\omega)\phi^{\mu}(x), \quad x \in \mathbb{R},
	$$
we have that $v(t, x)=\phi^{\mu}(x)$ is a super-solution of \eqref{3.1b} with $c(t ; \omega)=c(t ; \omega, \mu)$. We also note that $v(x, t) \equiv 1$ is a solution of \eqref{3.1b}. Therefore, by comparison principle, we complete the proof of Lemma \ref{Lem3.1}.
\end{proof} 
\begin{lemma}\label{Lem3.2}
	Suppose that (H1) holds. Let $\omega \in \Omega_{0}$. Then for every $0<\mu<$ $\tilde{\mu}<\min \left\{2 \mu, \mu^{*} \right\}$, there exist $\left\{t_{k}\right\}_{k \in \mathbb{Z}}$ with $t_{k}<t_{k+1}$ and $\lim _{k \rightarrow \pm \infty} t_{k}=\pm \infty$, $A_{\omega} \in W_{\text {loc }}^{1, \infty}(\R) \cap L^{\infty}(\R)$ with $A_{\omega}(\cdot) \in C^{1}\left(\left(t_{k}, t_{k+1}\right)\right)$ for $k \in \mathbb{Z}$, and a positive real number $d_{\omega}$ such that for every $d \geq d_{\omega}$ the function
	$$
	\phi^{\mu, d, A_{\omega}}(t, x):=e^{-\mu x}-d e^{\left(\frac{\tilde{\mu}}{\mu}-1\right) A_{\omega}(t)-\tilde{\mu}x},
	$$
	satisfies
	$$
	\mathcal{G}^{\omega, \mu}\left(\phi^{\mu, d, A_{\omega}}\right)(t, x) \leq 0 \quad \text { for } t \in\left(t_{k}, t_{k+1}\right), \quad x \geq \frac{\ln d}{\tilde{\mu}-\mu}+\frac{A_{\omega}(t)}{\mu}, \quad k \in \mathbb{Z},
	$$
	where $$
	\mathcal{G}^{\omega, \mu}(v)(t, x):=v_{t}-\int_{\R} J(\theta_{t}\omega,y)[v(t, x-y)-v(t, x)]\dy-c(t ; \omega, \mu) v_{x}-a\left(\theta_{t} \omega\right) v(1-v).
	$$
\end{lemma}
	\begin{proof}
For given $0<\mu<\tilde{\mu}<\min \left\{2 \mu, \mu^{*}\right\}$, it follows from Proposition \ref{Pro2.2} that
\begin{equation*}
\left\lfloor c(t;\omega,\mu)-c\left(t;\omega,\tilde{\mu}\right)\right\rfloor>0.		
\end{equation*}
Using Lemma \ref{Lem2.1}, there exist $T>0$, $\varepsilon>0$ and $A_{\omega} \in W_{\text {loc }}^{1, \infty}(\R) \cap L^{\infty}(\R)$ such that $A_{\omega} \in C^{1}\left(\left(t_{k}, t_{k+1}\right)\right)$, where $t_{k}=k T$ for $k \in \mathbb{Z}$ and
\begin{equation}\label{3.6}
\tilde{\mu} c(t ; \omega, \mu)-\int_{\R} J(\theta_{t}\omega,y)(\ee^{\tilde{\mu}y}-1)\dy-a(\theta_{t}\omega)+\left(\frac{\tilde{\mu}}{\mu}-1\right) A_{\omega}^{\prime}(t)\geq\varepsilon.
\end{equation}
We fix the above $\varepsilon>0$ and $A_{\omega}(t)$. Let $d>0$ to be determined later. 
By some computation, we have that
\begin{equation}\label{3.2}
	\phi^{\mu, d, A_{\omega}}_{t}(t,x)=-d\left(\frac{\tilde{\mu}}{\mu}-1\right) A_{\omega}^{\prime}(t) e^{\left(\frac{\tilde{\mu}}{\mu}-1\right) A_{\omega}(t)-\tilde{\mu}x},
\end{equation}
and
\begin{equation}\label{3.3}
	\phi^{\mu, d, A_{\omega}}_{x}(t,x)=-\mu \ee^{-\mu x}+\tilde{\mu}d e^{\left(\frac{\tilde{\mu}}{\mu}-1\right) A_{\omega}(t)-\tilde{\mu}x}.
\end{equation}
It follows from \eqref{3.2} and \eqref{3.3} that for any $t \in\left(t_{k}, t_{k+1}\right)$
\begin{equation}\label{3.4}
\begin{aligned}
&\mathcal{G}^{\omega, \mu}\left(\phi^{\mu, d, A_{\omega}}\right)(t, x)\\
&=\phi^{\mu, d, A_{\omega}}_{t}(t,x)-\int_{\R} J(\theta_{t}\omega,y)[\phi^{\mu, d, A_{\omega}}(t, x-y)-\phi^{\mu, d, A_{\omega}}(t, x)]\dy-c(t; \omega, \mu) \phi^{\mu, d, A_{\omega}}_{x}(t,x)\\
&\quad-a\left(\theta_{t} \omega\right) \phi^{\mu, d, A_{\omega}}(t,x)(1-\phi^{\mu, d, A_{\omega}}(t,x))\\
&=d e^{\left(\frac{\tilde{\mu}}{\mu}-1\right) A_{\omega}(t)-\tilde{\mu}x} \Bigg[-\left(\frac{\tilde{\mu}}{\mu}-1\right) A_{\omega}^{\prime}(t)+\int_{\R} J(\theta_{t}\omega,y)(\ee^{\tilde{\mu}y}-1)\dy-\tilde{\mu} c(t ; \omega, \mu)+a(\theta_{t}\omega)\Bigg] \\
&\quad-\ee^{-\mu x}\left[\int_{\R} J(\theta_{t}\omega,y)(\ee^{\mu y}-1)\dy
-\mu c(t ; \omega, \mu)+a(\theta_{t}\omega)\right]+a(\theta_{t}\omega)(\phi^{\mu, d, A_{\omega}})^2(t, x).
\end{aligned}
\end{equation}		
Recall that 
\begin{equation*}
	c(t;\omega,\mu):=\mu^{-1}\left[\int_{\mathbb{R}} J(\theta_{t}\omega, y)\left(\mathrm{e}^{\mu y}-1\right) \mathrm{d} y+a(\theta_{t}\omega)\right].
\end{equation*}
Thus we have that for $t \in\left(t_{k}, t_{k+1}\right)$
\begin{equation}\label{3.5}
\begin{aligned}
&\mathcal{G}^{\omega, \mu}\left(\phi^{\mu, d, A_{\omega}}\right)(t, x)\\
&=d e^{\left(\frac{\tilde{\mu}}{\mu}-1\right) A_{\omega}(t)-\tilde{\mu}x} \left[-\left(\frac{\tilde{\mu}}{\mu}-1\right) A_{\omega}^{\prime}(t)+\int_{\R} J(\theta_{t}\omega,y)(\ee^{\tilde{\mu}y}-1)\dy
-\tilde{\mu} c(t ; \omega, \mu)+a(\theta_{t}\omega)\right]\\
&\quad+a(\theta_{t}\omega)(\phi^{\mu, d, A_{\omega}})^2(t, x).
\end{aligned}
\end{equation}
On the one hand, by direct calculation, we have
\begin{equation}\label{3.6a}
\phi^{\mu, d, A_{\omega}}(t, x)\geq0,\text { for }x \geq \frac{\ln d}{\tilde{\mu}-\mu}+\frac{A_{\omega}(t)}{\mu}, t \in\left(t_{k}, t_{k+1}\right), k \in \mathbb{Z}.
\end{equation}	
On the other hand, we have that
\begin{equation}\label{3.7}
\phi^{\mu, d, A_{\omega}}(t, x)\leq e^{-\mu x}
\leq  e^{-\mu\left(\frac{\ln d}{\tilde{\mu}-\mu}+\frac{A_{\omega}(t)}{\mu}\right)},\text { for }x \geq \frac{\ln d}{\tilde{\mu}-\mu}+\frac{A_{\omega}(t)}{\mu}, t \in\left(t_{k}, t_{k+1}\right), k \in \mathbb{Z}.
\end{equation}	
Using \eqref{3.6}, \eqref{3.5}, \eqref{3.6a} and \eqref{3.7}, we have that
\begin{equation}\label{3.8}
\begin{aligned}
\mathcal{G}^{\omega, \mu}\left(\phi^{\mu, d, A_{\omega}}\right)(t, x)
&\leq -d\varepsilon e^{\left(\frac{\tilde{\mu}}{\mu}-1\right) A_{\omega}(t)-\tilde{\mu}x} +a(\theta_{t}\omega)e^{-2\mu x} \\
&\leq \ee^{\left(\frac{\tilde{\mu}}{\mu}-1\right) A_{\omega}(t)-\tilde{\mu}x}
\left[-d\varepsilon+a(\theta_{t}\omega)\ee^{(\tilde{\mu}-2\mu)x-\left(\frac{\tilde{\mu}}{\mu}-1\right)A_{\omega}(t)}\right] \\
&\leq\ee^{\left(\frac{\tilde{\mu}}{\mu}-1\right) A_{\omega}(t)-\tilde{\mu}x}\left[-d\varepsilon+d^{\left(1-\dfrac{\mu}{\tilde{\mu}-\mu}\right)} \sup_{t\in\R}a(\theta_{t}\omega)\ee^{-A_{\omega}(t)}\right] \\
&=d\ee^{\left(\frac{\tilde{\mu}}{\mu}-1\right) A_{\omega}(t)-\tilde{\mu}x}\left[-\varepsilon+d^{-\dfrac{\mu}{\tilde{\mu}-\mu}} \sup_{t\in\R}a(\theta_{t}\omega)\ee^{-A_{\omega}(t)}\right] \\
&\leq0.
\end{aligned}
\end{equation}		
The last inequality above holds by choosing 
$$
d\geq \left(\dfrac{\sup_{t\in\R}a(\theta_{t}\omega)}{\varepsilon\ee^{\vert\vert A_{\omega}\vert\vert_{\infty}}}\right)^{\dfrac{\tilde{\mu}-\mu}{\mu}}.
$$		
Thus we complete the proof of the lemma.
\end{proof}
Let $0<\mu<\tilde{\mu}<\min \left\{2 \mu, \mu^{*}\right\}$ be given. Let $A_{\omega}$ and $d_{\omega}$ be given by Lemma \ref{Lem3.2}. Let
\begin{equation}\label{3.9}
x_{\omega}(t)=\frac{\ln d_{\omega}+\ln \tilde{\mu}-\ln \mu}{\tilde{\mu}-\mu}+\frac{A_{\omega}(t)}{\mu}.
\end{equation}
Note that for any given $t \in \mathbb{R}$,
	$$
	\phi^{\mu, d_{\omega}, A_{\omega}}\left(t, x_{\omega}(t)\right)=\sup _{x \in \mathbb{R}} \phi^{\mu, d_{\omega}, A_{\omega}}(t, x)=e^{-\mu\left(\frac{\ln d}{\mu-\mu}+\frac{A_{\omega}(t)}{\mu}\right)} e^{-\mu \frac{\ln \tilde{\mu} \ln \mu}{\mu-\mu}}\left(1-\frac{\mu}{\tilde{\mu}}\right) .
	$$
	We introduce the following function:
	$$
	\phi_{-}^{\mu}\left(t, x ; \theta_{t_{0}} \omega\right)= \begin{cases}\phi^{\mu, d_{\omega}, A_{\omega}}\left(t+t_{0}, x\right), & \text { if } x \geq x_{\omega}\left(t+t_{0}\right), \\ \phi^{\mu, d_{\omega}, A_{\omega}}\left(t+t_{0}, x_{\omega}\left(t+t_{0}\right)\right), & \text { if } x \leq x_{\omega}\left(t+t_{0}\right) .\end{cases}
	$$
	It is clear that
	$$
	0<\phi_{-}^{\mu}\left(t, x ; \theta_{t_{0}} \omega\right)<\phi_{+}^{\mu}(x) \leq 1, \quad \forall t \in \mathbb{R}, x \in \mathbb{R}, t_{0} \in \mathbb{R}
	$$
	and
\begin{equation}\label{3.10}
\lim _{x \rightarrow \infty} \sup _{t>0, t_{0} \in \mathbb{R}}\left|\frac{\phi_{-}^{\mu}\left(t, x; \theta_{t_{0}} \omega\right)}{\phi_{+}^{\mu}(x)}-1\right|=0.
\end{equation}

\subsection{Construction of a solution by a limiting procedure}

For any integer $n \geq 1$, we consider the following Cauchy problem, for $t \geq-n$ and $x \in \mathbb{R}$,
\begin{equation}\label{3.10a}
	\left\{\begin{array}{l}
		\partial_{t}v(t, x)=c(t; \omega,\mu)\partial_{x}v(t, x)+\int_{\R} J(\theta_{t}\omega,y)[v(t, x-y)-v(t, x)]\dy+a(\theta_{t}\omega)v(t, x)(1-v(t, x)),\\
		v(-n, x)=\phi^{\mu}_{+}(-n, x).
	\end{array}\right.
\end{equation}
We denote by $v(t, x;\phi^{\mu}_{+},\theta_{-n}\omega)$ the solution of \eqref{3.10a} and define the function $u(t, x,\phi^{\mu}_{+},\theta_{-n}\omega)$ by
\begin{equation}\label{3.11c}
	u(t, x,\phi^{\mu}_{+},\theta_{-n}\omega)=v\left(t, x-C(t;\theta_{-n}\omega,\mu),\phi^{\mu}_{+},\theta_{-n}\omega\right).
\end{equation}
It follows from \eqref{3.10a} and \eqref{3.11c} that the function $u(t, x,\phi^{\mu}_{+},\theta_{-n}\omega)$ satisfies the following equation
\begin{equation}\label{4.1}
	\left\{\begin{array}{l}
		\partial_{t} u(t, x)=\int_{\R} J(\theta_{t}\omega, y)[u(t, x-y)-u(t, x)] \mathrm{d}y+a(\theta_{t}\omega)u(t, x)(1-u(t, x)), t \geq-n, x \in \mathbb{R}, \\
		u(-n, x)=\phi^{\mu}_{+}(-n, x-C(t;\theta_{-n}\omega,\mu)), x \in \mathbb{R}.
	\end{array}\right.
\end{equation}
It follows from Lemmas \ref{Lem3.1}, \ref{Lem3.2} and comparison principle stated in Proposition \ref{Prop2.4} that
the solution $u(t, x,\phi^{\mu}_{+},\theta_{-n}\omega)$ of \eqref{4.1} for all $t \geq-n$ and $x \in \mathbb{R}$ satisfies 
	$$
\phi_{-}^{\mu}(t, x-C(t;\theta_{-n}\omega,\mu); \theta_{-n}\omega) \leq u(t, x,\phi^{\mu}_{+},\theta_{-n}\omega)\leq \phi_{+}^{\mu}(x-C(t;\theta_{-n}\omega,\mu)), \quad \forall t>0, x \in \mathbb{R} .
$$
Furthermore, since the function $x \mapsto \phi_{+}^{\mu}( x-C(t;\theta_{-n}\omega,\mu))$ is nonincreasing on $\mathbb{R}$, the function $x \mapsto u(t, x,\phi^{\mu}_{+},\theta_{-n}\omega)$ is also nonincreasing with respect to $x \in \R$ for each given $t \geq-n$.

Our aim now is to pass to the limit $n \rightarrow \infty$ in the sequence of function $\left\{u(t, x;\phi^{\mu}_{+},\theta_{-n}\omega)\right\}$ to construct a random transition front of system \eqref{1.1}. To achieve our aim, we first discuss in the following lemma some important Lipschite regularity estimates, inspired by \cite{Ducrot21,Shen15}.
\begin{lemma}\label{Lem4.1}
	There exists some constant $m>\mu$ large enough such that for all $n \geq 1$ one has
	$$
	\left|u(t, x+h,\phi^{\mu}_{+},\theta_{-n}\omega) -u(t, x,\phi^{\mu}_{+},\theta_{-n}\omega)\right| \leq \min \left\{1, \mathrm{e}^{m|h|}-1\right\}, \forall t \geq-n, \forall x \in \mathbb{R}.
	$$
	For all $n \geq 1$ one has $\partial_{t}u(t, x,\phi^{\mu}_{+},\theta_{-n}\omega) \in L^{\infty}((-n, \infty) \times \mathbb{R}\times\Omega_0)$ and the following estimate holds
	$$
	\left\|\partial_{t}u(t, x,\phi^{\mu}_{+},\theta_{-n}\omega)\right\|_{\infty} \leq 2 \int_{\mathbb{R}}\|\left. J(\cdot, y)\right\|_{\infty} \mathrm{d} y+1, \forall n \geq 1.
	$$
	In other words, the sequence $\left\{u(t, x,\phi^{\mu}_{+},\theta_{-n}\omega)\right\}$ is uniformly bounded (with respect to $n$ ) in the Lipschitz norm on the set $[-n, \infty) \times \mathbb{R}\times\Omega_{0}$.
\end{lemma}
\begin{proof}
The proof is similar to the proof of Lemma 5.1 in \cite{Ducrot21}, so we omit it.	
\end{proof}

By using Lemma \ref{Lem4.1} and Arzelà-Ascoli theorem, there exists a subsequence of $\{u(t, x,\phi^{\mu}_{+},$
$\theta_{-n}\omega)\}$, still denoted with the same indexes, and a globally Lipschitz continuous function $u(t, x;\phi^{\mu}_{+},\omega): \mathbb{R}^{2}\times\Omega_{0} \rightarrow \mathbb{R}$ such that
\begin{equation}\label{3.13a}
	\lim_{n\to+\infty} u(t, x,\phi^{\mu}_{+},\theta_{-n}\omega) =u(t, x;\phi^{\mu}_{+},\omega)\text{ locally uniformly for } (t, x) \in \mathbb{R}^{2}, \omega\in\Omega_0.
\end{equation}
Therefore, we can define the Lipschitz continuous function
$v(t, x,\phi^{\mu}_{+},\omega)$ by
\begin{equation}\label{3.14a}
	v\left(t, x,\phi^{\mu}_{+},\omega\right)=u(t, x+C(t;\omega,\mu),\phi^{\mu}_{+},\omega),\forall(t, z) \in \mathbb{R}^{2},\omega\in\Omega_0.
\end{equation}

Based on Lemmas \ref{Lem3.1}-\ref{Lem4.1}, we summarize the
following propositions, which gives some important properties satisfied by the function $v(t, x,\phi^{\mu}_{+},\omega)$. 
\begin{proposition}\label{Pro4.2}
The function $v(t, x,\phi^{\mu}_{+},\omega)$ enjoys the following properties.
\begin{description}
\item[(i)] It is nonincreasing with respect to $x \in \mathbb{R}$, for all $t \in \mathbb{R}$, $\omega\in\Omega_{0}$, and is globally Lipschitz continuous on $\mathbb{R}^{2}\times\Omega_{0}$;
\item[(ii)] It satisfies the following estimate for all $(t, x) \in \mathbb{R}^{2}$, $\omega\in\Omega_{0}$
$$
\phi_{-}^{\mu}\left(t, x; \omega\right)\leq v(t, x,\phi^{\mu}_{+},\omega)\leq\phi_{+}^{\mu}(x), \quad \forall t >0, x \in \mathbb{R}.
$$
\item[(iii)] It satisfies \eqref{3.1b} with $c(t;\omega,\mu)=C'(t;\omega,\mu)$ for any $x \in \mathbb{R}$ and for a.e. $t \in \mathbb{R}$ and $\omega\in\Omega_0$.
\end{description}
\end{proposition}
Hence to complete the proof of Theorem \ref{Th2.3}, it remains to study the behavior of $v(t, x,\theta_{t}\omega)$ as $x \rightarrow\pm\infty$, that is, $u(t, x+C\left(t, \omega, \mu\right),\theta_{t}\omega)$ as $x \rightarrow\pm\infty$. We first prove the behavior of $u(t, x+C\left(t, \omega, \mu\right),\theta_{t}\omega)$ as $x \rightarrow-\infty$, see Lemma \ref{Lem4.3}.
\begin{lemma}\label{Lem4.3}
	For every $\omega\in\Omega_{0},$
	\begin{equation*}
		\lim_{x\rightarrow-\infty}	u\left(t, x+C\left(t, \theta_{t_{0}} \omega, \mu\right) ; \phi_{+}^{\mu}\left(0, \cdot ; \theta_{t_{0}} \omega\right), \theta_{t_{0}} \omega\right)=1
	\end{equation*}
	uniformly in $t>0$ and $t_0\in\R$.
\end{lemma}
\begin{proof}
Define 
\begin{equation}\label{4.3a}
v\left(t, x; \theta_{t_{0}} \omega\right)=u\left(t, x+C\left(t, \theta_{t_{0}} \omega, \mu\right) ; \phi_{+}^{\mu}\left(0, \cdot ; \theta_{t_{0}} \omega\right), \theta_{t_{0}} \omega\right),
\end{equation}
and
$$
x^{*}=\frac{\ln d_{\omega}+\ln \tilde{\mu}-\ln \mu}{\tilde{\mu}-\mu}-\frac{\left\|A_{\omega}\right\|_{\infty}}{\mu}.
$$
It follows from \eqref{3.9} and Proposition \ref{Pro4.2} that
\begin{equation}\label{4.3}
0<\left(1-\frac{\mu}{\tilde{\mu}}\right) e^{-\mu\left(\frac{\ln d_{\omega}+\ln \tilde{\mu}-\ln \mu}{\tilde{\mu}-\mu}+\frac{\left\|A_{\omega}\right\| \infty}{\mu}\right)} \leq \inf _{t>0, t_{0} \in \mathbb{R}} v\left(t, x^{*} ; \phi_{+}^{\mu}\left(0, \cdot ; \theta_{t_{0}} \omega\right), \theta_{t_{0}} \omega\right).
\end{equation}
Moreover, $x \mapsto v\left(t, x ; \theta_{t_{0}} \omega\right)$ is decreasing and satisfies
$$
v_{t}\left(t, x\right)=\int_{\R} J(\theta_{t}\theta_{t_0}\omega,y)[v(t, x-y)-v(t, x)]\dy+c\left(t ; \theta_{t_{0}} \omega, \mu\right) v_{x}+a\left(\theta_{t} \theta_{t_{0}} \omega\right) v(1-v),
$$
where $c\left(t ; \theta_{t_{0}} \omega, \mu\right)=C^{\prime}\left(t ; \theta_{t_{0}} \omega, \mu\right)$. It follows from \eqref{4.3} that 
\begin{equation*}
\inf _{t>0, t_{0} \in \mathbb{R}} v\left(t, x^{*} ; \phi_{+}^{\mu}\left(0, \cdot ; \theta_{t_{0}} \omega\right), \theta_{t_{0}} \omega\right)>0.
\end{equation*}
Therefore, by using \eqref{4.3a}, we also have that
$$
\Theta:=\lim _{x \rightarrow-\infty} \inf _{t>0,t_0 \in \mathbb{R}} u\left(t, x+C\left(t, \theta_{t_{0}} \omega, \mu\right) ; \phi_{+}^{\mu}\left(0, \cdot ; \theta_{t_{0}} \omega\right), \theta_{t_{0}} \omega\right)>0.
$$
Since $u \leq 1$, to prove the lemma, it is sufficient to check that $\Theta=1$. Indeed, we consider two sequences $\left\{t_{n}\right\} \subset \mathbb{R}$ and $\left\{x_{n}\right\} \subset \mathbb{R}$ such that $x_{n} \rightarrow-\infty$ as $n \rightarrow \infty$ and
$$
\lim _{n \rightarrow \infty} u\left(t_n, x_n+C\left(t_n, \theta_{t_{0}} \omega, \mu\right) ; \phi_{+}^{\mu}\left(0, \cdot ; \theta_{t_{0}} \omega\right), \theta_{t_{0}} \omega\right)=\Theta .
$$
We consider the sequence of functions $\left\{u_n=u_{n}(t, x,\phi_{+}^{\mu}\left(0, \cdot ; \theta_{t_{0}} \omega\right), \theta_{t_{0}} \omega) \right\}$ given for $n \geq 1$ by
\begin{equation}\label{3.17}
	u_{n}(t, x,\phi_{+}^{\mu}\left(0, \cdot ; \theta_{t_{0}} \omega\right), \theta_{t_{0}} \omega) =u(t+t_n, x+x_n+C\left(t_n, \theta_{t_{0}} \omega, \mu\right); \phi_{+}^{\mu}\left(0, \cdot ; \theta_{t_{0}} \omega\right), \theta_{t_{0}} \omega),
\end{equation}
with
\begin{equation}\label{3.18}
	C\left(t_n, \theta_{t_{0}} \omega, \mu\right) =\int_{0}^{t_{n}} c\left(\tau, \theta_{t_{0}} \omega, \mu\right) \mathrm{d}\tau.
\end{equation}
It is clear that the sequence $\left\{u_{n}\right\}$ is uniformly bounded in the Lipschitz norm on $\mathbb{R}^{2}$, therefore there exists a subsequence such that for $\omega\in\Omega_0$
$$
u_{n}(t, x,\phi_{+}^{\mu}\left(0, \cdot ; \theta_{t_{0}} \omega\right), \theta_{t_{0}} \omega) \rightarrow u_{\infty}(t, x,\phi_{+}^{\mu}\left(0, \cdot ; \theta_{t_{0}} \omega\right), \theta_{t_{0}} \omega) \text{ locally uniformly for }(t, x) \in \mathbb{R}^{2}, 
$$
and
$$
u_{\infty}(0,0,\phi_{+}^{\mu}\left(0, \cdot ; \theta_{t_{0}} \omega\right), \theta_{t_{0}} \omega)=\Theta.
$$
We now claim that
\begin{equation}\label{4.5m}
u_{\infty}(t, x,\phi_{+}^{\mu}\left(0, \cdot ; \theta_{t_{0}} \omega\right), \theta_{t_{0}} \omega) \geq \Theta, \forall(t, z) \in \mathbb{R}^{2},\omega\in\Omega_0.
\end{equation}
Indeed, by \eqref{3.17} and \eqref{3.18}, we have that
\begin{equation*}
\begin{aligned}
&u_{n}(t, x,\phi_{+}^{\mu}\left(0, \cdot ; \theta_{t_{0}} \omega\right), \theta_{t_{0}} \omega)\\
&=u\left(t+t_{n}, x+x_{n}-\int_{0}^{t} c\left(t_n +s, \theta_{t_{0}} \omega, \mu\right)\mathrm{d}s+\int_{0}^{t+t_{n}} c\left(s, \theta_{t_{0}} \omega, \mu\right) \mathrm{d} s\right).
\end{aligned}
\end{equation*}
Now, since one has locally uniformly for $(t, x) \in \mathbb{R}^{2}$, $\omega\in\Omega_0,$
$$
x+x_{n}-\int_{0}^{t} c\left(t_n +s, \theta_{t_{0}} \omega, \mu\right)\mathrm{d}s  \rightarrow-\infty,
$$
it follows from the definition of $\Theta$ that \eqref{4.5m} holds true.
	
Now we derive the equation satisfied by $u_{\infty}$. Since the function $u\left(t, x; \phi_{+}^{\mu}\left(0, \cdot ; \theta_{t_{0}} \omega\right), \theta_{t_{0}} \omega\right)$ satisfies the following equation for all $(t, x) \in \mathbb{R}^{2}$,
$$
u_{t}=\int_{\R} J(\theta_{t}\theta_{t_0}\omega,y)[u(t, x-y)-u(t, x)]\dy+a\left(\theta_{t} \theta_{t_{0}} \omega\right) u(1-u),
$$
we have that for any $n \geq 1$ the function $u_{n}$ satisfies the shifted equation 
$$
\partial_{t} u_{n}=\int_{\R} J(\theta_{t+t_n}\theta_{t_0}\omega,y)[u_{n}(t, x-y)-u_{n}(t, x)]\dy+a\left(\theta_{t+t_n} \theta_{t_{0}} \omega\right) u_{n}(1-u_{n}).
$$
In order to obtain a suitable equation for $u_{\infty}$, we first investigate the shifted kernel function $(t, y) \mapsto J(\theta_{t+t_n}\theta_{t_0}\omega,y)$. Note that $y \mapsto J(\cdot, y) \in L^{1}\left(\mathbb{R} ; L^{\infty}(\Omega,\mathcal{F},\mathbb{P})\right)$, then by applying Dunford-Pettis theorem, we have that the sequence $\left\{(t, y) \mapsto J(\theta_{t+t_n}\theta_{t_0}\omega,y)\right\}$ is relatively weakly compact in $L^{1}((-T, T) \times \mathbb{R}; L^{\infty}(\Omega,\mathcal{F},\mathbb{P}))$ for any $T>0$. Therefore, there exist a subsequence and a function $\bar{J}=\bar{J}(\theta_{t}\theta_{t_0}\omega, y) \in L_{l o c}^{1}\left(\R\times\mathbb{R}; L^{\infty}(\Omega,\mathcal{F},\mathbb{P})\right)$ with
$$
0 \leq \bar{J}(\theta_{t}\theta_{t_0}\omega, y) \leq\|J(\cdot, y)\|_{L^{\infty}(\R)}, \quad a.e. (t, y) \in \mathbb{R}^{2},\omega\in\Omega_{0}
$$
and such that for all $T>0$ and any $\varphi \in L^{\infty}((-T, T) \times \mathbb{R})$ the following convergence holds
$$
\lim _{n \rightarrow \infty} \int_{-T}^{T} \int_{\R} J\left(\theta_{t+t_{n}}\theta_{t_0}\omega, y\right) \varphi(t, y) \mathrm{d} t \mathrm{d} y=\int_{-T}^{T} \int_{\mathbb{R}} \bar{J}(\theta_{t}\theta_{t_0}\omega, y) \varphi(t, y) \mathrm{d} t \mathrm{d} y .
$$
Taking $\varphi(t, y) \equiv 1$, we obtain that
\begin{equation}\label{3.20}
	\int_{\mathbb{R}}J\left(\theta_{t+t_{n}}\theta_{t_0}\omega, y\right) \dy \rightarrow \int_{\mathbb{R}}\bar{J}(\theta_{t}\theta_{t_0}\omega, y) \mathrm{d} y \text { weakly in } L_{l o c}^{1}(\mathbb{R}),
\end{equation}
therefore
$$
u_{n}(t, x) \int_{\mathbb{R}} J \left(\theta_{t+t_{n}}\theta_{t_0}\omega, y\right) \mathrm{d} y \rightarrow u_{\infty}(t, x) \int_{\mathbb{R}} 
\bar{J}(\theta_{t}\theta_{t_0}\omega, y) \mathrm{d} y,
$$
weakly in $L_{\text {loc }}^{1}(\R)$ with respect to $t$ and locally uniformly with respect to $x \in \mathbb{R}$.
\end{proof}
Applying the above convergence to the sequence of shifted kernels, we now claim 
\begin{claim}\label{Cla4.4}
	The following holds
	$$
	\lim _{n \rightarrow \infty} \int_{\mathbb{R}} J\left(\theta_{t+t_{n}}\theta_{t_0}\omega, y\right) u_{n}(t, x-y) \mathrm{d} y=\int_{\mathbb{R}} 
	\bar{J}(\theta_{t}\theta_{t_0}\omega, y) u_{\infty}(t, x-y) \mathrm{d} y,
	$$
	weakly $L_{l \mathrm{oc}}^{1}(\mathbb{R})$ with respect to $t$ and locally uniformly with respect to $x \in \mathbb{R}$. In other words, for any $T>0$ and any $\psi \in L^{\infty}(-T, T)$ one has
	$$
	\lim _{n \rightarrow \infty} \int_{-T}^{T} \int_{\mathbb{R}} \psi(t) J\left(\theta_{t+t_{n}}\theta_{t_0}\omega, y\right) u_{n}(t, x-y) \mathrm{d} t \mathrm{d} y=\int_{-T}^{T} \int_{\mathbb{R}} \psi(t) 
	\bar{J}(\theta_{t}\theta_{t_0}\omega, y) u_{\infty}(t, x-y) \mathrm{d} d \mathrm{d} y
	$$
	locally uniformly with respect to $x \in \mathbb{R}$.
\end{claim}
\begin{proof}
It follows from \eqref{3.20} that
\begin{equation}\label{4.6m}
\lim_{ n \rightarrow \infty}\int_{\mathbb{R}} J\left(\theta_{t+t_{n}}\theta_{t_0}\omega, y\right) u_{\infty}(t, x-y) \mathrm{d} y=\int_{\mathbb{R}} 
\bar{J}(\theta_{t}\theta_{t_0}\omega, y) u_{\infty}(t, x-y) \mathrm{d} y,
\end{equation}
locally uniformly for $x \in \mathbb{R}$ and weakly in $L_{l o c}^{1}(\R)$ with respect to the $t$. It is clearly that for any $n$ one has
\begin{equation}\label{3.22}
	\begin{aligned}
		&\int_{\mathbb{R}} J\left(\theta_{t+t_{n}}\theta_{t_0}\omega, y\right) u_{n}(t, x-y) \mathrm{d} y-\int_{\mathbb{R}} 
		\bar{J}(\theta_{t}\theta_{t_0}\omega, y) u_{\infty}(t, x-y) \mathrm{d} y \\
		=& \int_{\mathbb{R}} J\left(\theta_{t+t_{n}}\theta_{t_0}\omega, y\right)\left[u_{n}(t, x-y)-u_{\infty}(t, x-y)\right] \mathrm{d} y \\
		&+\int_{\mathbb{R}}\left[J\left(\theta_{t+t_{n}}\theta_{t_0}\omega, y\right)-
		\bar{J}(\theta_{t}\theta_{t_0}\omega, y)\right] u_{\infty}(t, x-y) \mathrm{d} y.
	\end{aligned}
\end{equation}
Since \eqref{4.6m} and \eqref{3.22}, in order to prove Claim \ref{Cla4.4} we only need to prove that
\begin{equation}\label{4.7}
\int_{\mathbb{R}} J\left(\theta_{t+t_{n}}\theta_{t_0}\omega, y\right)\left[u_{n}(t, x-y)-u_{\infty}(t, x-y)\right] \mathrm{d} y \rightarrow 0 \text { as } n \rightarrow \infty,
\end{equation}
locally uniformly for $(t, x) \in \mathbb{R}^{2}$. Note that for any $B>0$ one has
$$
\begin{aligned}
&\left| \int_{\mathbb{R}} J\left(\theta_{t+t_{n}}\theta_{t_0}\omega, y\right)\left[u_{n}(t, x-y)-u_{\infty}(t, x-y)\right] \mathrm{d} y\right| \\
\leq & \int_{|y| \leq B} J\left(\theta_{t+t_{n}}\theta_{t_0}\omega, y\right)\left|u_{n}(t, x-y)-u_{\infty}(t, x-y)\right| \mathrm{d} y \\	&+\int_{|y| \geq B} J\left(\theta_{t+t_{n}}\theta_{t_0}\omega, y\right)\left|u_{n}(t, x-y)-u_{\infty}(t, x-y)\right| \mathrm{d} y .
\end{aligned}
$$
On the one hand, since $0 \leq u_{n} \leq 1$ the above inequality implies that for all $B>0$, any $n$ and any $(t, x) \in \mathbb{R}^{2}$ one has
$$
\begin{aligned}
&\left| \int_{\mathbb{R}} J\left(\theta_{t+t_{n}}\theta_{t_0}\omega, y\right)\left[u_{n}(t, x-y)-u_{\infty}(t, x-y)\right] \mathrm{d} y\right|\\
& \leq \int_{\mathbb{R}}\|J(\cdot, y)\|_{L^{\infty}(\mathbb{R})} \mathrm{d} y \sup _{|y| \leq B}\left|u_{n}(t, x-y)-u_{\infty}(t, x-y)\right|+2 \int_{|y| \geq B}\|J(\cdot, y)\|_{L^{\infty}(\mathbb{R})} \mathrm{d} y.
\end{aligned}
$$
On the other hand, since $u_{n}(t, x) \rightarrow u_{\infty}(t, x)$ locally uniformly for $(t, x) \in \mathbb{R}^{2}$, we obtain that for each $A>0$ and any $B>0$
\begin{equation}
	\begin{aligned}
		&\limsup _{n \rightarrow \infty} \sup _{(t, x) \in[-A, A]^{2}} \left|\int_{\mathbb{R}} J\left(\theta_{t+t_{n}}\theta_{t_0}\omega, y\right)\left[u_{n}(t, x-y)-u_{\infty}(t, x-y)\right] \mathrm{d} y\right|\\
		& \leq 2 \int_{|y| \geq B}\|J(\cdot, y)\|_{L^{\infty}(\mathbb{R})} \mathrm{d} y .
	\end{aligned}
\end{equation}
	Finally since $y \mapsto\|J(\cdot, y)\|_{L^{\infty}(\mathbb{R})} \in L^{1}(\mathbb{R})$, letting $B \rightarrow \infty$ ensures that \eqref{4.7} holds and this completes the proof of Claim \ref{Cla4.4}.
\end{proof}

Next we consider the sequence of function $h_{n}(t, x,\theta_{t_{0}}\omega)=a\left(\theta_{t+t_n} \theta_{t_{0}} \omega\right) (1-u_{n})$. By Assumption (H1) and $0\leq u_n\leq1$, we know that $h_{n}(t, x,\theta_{t_{0}}\omega)$ is a bounded sequence in $L^{\infty}\left(\mathbb{R}^{2}\right)$. Then up to a subsequence, one may assume that it converges for the weak-$\star$ topology of $L^{\infty}\left(\mathbb{R}^{2}\right)$ to some function $h_{\infty}=h_{\infty}(t, x,\theta_{t_{0}}\omega) \in L^{\infty}\left(\mathbb{R}^{2}\right)$. Using Assumption (H1) again, the function $h_{\infty}$ satisfies
\begin{equation}\label{4.8}
	\inf_{t\in\R}a\left(\theta_{t} \theta_{t_{0}} \omega\right) (1-u_{\infty})\leq h_{\infty}(t, x,\omega)\leq \sup_{t\in\R}a\left(\theta_{t} \theta_{t_{0}} \omega\right) (1-u_{\infty}).
\end{equation}
As a consequence the Lipschitz continuous function $u_{\infty}$ satisfies the equation for a.e. $(t, x) \in \mathbb{R}^{2}$, $\omega\in\Omega_{0}$
$$
\partial_{t} u_{\infty}(t, x)=\int_{\mathbb{R}} 
\bar{J}(\theta_{t}\theta_{t_0}\omega, y)\left[u_{\infty}(t, x-y)-u_{\infty}(t, x)\right] \mathrm{d} y+u_{\infty}(t, x) h_{\infty}(t, x,\omega),
$$
together with $0<\Theta \leq u_{\infty}(t, x) \leq 1$ for all $(t, x) \in \mathbb{R}^{2}$ and $u_{\infty}(0, 0,\phi_{+}^{\mu}\left(0, \cdot ; \theta_{t_{0}} \omega\right), \theta_{t_{0}} \omega)=\Theta$.
Now let us complete the proof of the lemma by showing that $\Theta=1$. We consider the function $U=U(\theta_{t} \theta_{t_{0}} \omega)$ defined for $t \geq 0$,$t_0\in\R$ and $\omega\in\Omega_{0}$ by
$$
U^{\prime}(\theta_{t} \theta_{t_{0}} \omega)=\inf_{t\in\R}a\left(\theta_{t} \theta_{t_{0}} \omega\right) (1-U(\theta_{t} \theta_{t_{0}} \omega)) U(\theta_{t} \theta_{t_{0}} \omega), \forall t \geq 0 \text { and } U(\theta_{t_{0}} \omega)=\Theta.
$$
Then by using \eqref{4.8} and the comparison principle, we have that
$$
U(\theta_{t} \theta_{t_{0}} \omega) \leq u_{\infty}(s+t, x) \leq 1, \forall t \geq 0, \forall s \in \mathbb{R}, \forall x \in \mathbb{R}.
$$
As a consequence, we obtain that
$$
U(\theta_{t} \theta_{t_{0}} \omega) \leq u_{\infty}(0, 0,\phi_{+}^{\mu}\left(0, \cdot ; \theta_{t_{0}} \omega\right), \theta_{t_{0}} \omega)=\Theta \leq 1, \forall t \geq 0.
$$
Since $\Theta>0, U(\theta_{t} \theta_{t_{0}} \omega) \rightarrow 1$ as $t \rightarrow \infty$. This implies that $\Theta=1$ and completes the proof of the lemma.

Now, we are ready to prove the Theorem \ref{Th2.3}.
\begin{proof}[Proof of Theorem \ref{Th2.3}(1)]
First, let $0<\mu<\tilde{\mu}<\min \left\{2 \mu, \mu^{*}\right\}$ be fixed. 
It follows from \eqref{3.14a} and Proposition \ref{Pro4.2}
that 
$$
\begin{array}{r}
u\left(t, x+C(t ; \omega, \mu) ; \phi_{\mu}^{+}, \omega\right)<u\left(\tilde{t}, x+C(\tilde{t} ; \omega, \mu) ; \phi_{\mu}^{+}, \omega\right), \\
\forall x \in \mathbb{R}, \quad t>\tilde{t}>0, \forall \omega \in \Omega_{0}.
\end{array}
$$
Hence the following limit exits:
\begin{equation}\label{4.2}
	U^{\mu}(x, \omega):=\lim _{t \rightarrow \infty} u\left(t, x+C\left(t ; \theta_{-t} \omega, \mu\right) ; \phi_{\mu}^{+}, \theta_{-t} \omega\right), \quad \forall x \in \mathbb{R}, \omega \in \Omega_{0}.
\end{equation}
It follows from Lemma \ref{Lem3.1} that
$$
\sup _{t>0, t_{0} \in \mathbb{R}} u\left(t, x+C\left(t, \theta_{t_{0}} \omega, \mu\right) ; \phi_{+}^{\mu}\left(0, \cdot ; \theta_{t_{0}} \omega\right), \theta_{t_{0}} \omega\right) \leq e^{-\mu x} \rightarrow 0 \text { as } x \rightarrow \infty.
$$
Therefore, we have that
$$
\lim _{x \rightarrow+\infty} U^{\mu}\left(x, \theta_{t} \omega\right)=0, \text{ uniformly in }t \in \mathbb{R}.
$$
By the arguments of Lemma \ref{Lem4.3}, we have that
\begin{equation*}
	\lim_{x\rightarrow-\infty}	u\left(t, x+C\left(t, \theta_{t_{0}} \omega, \mu\right) ; \phi_{+}^{\mu}\left(0, \cdot ; \theta_{t_{0}} \omega\right), \theta_{t_{0}} \omega\right)=1,\text{ uniformly in }t>0, t_{0} \in \mathbb{R}, \omega\in\Omega_{0}.
\end{equation*}
Therefore, we have that
$$
\lim _{x \rightarrow-\infty} U^{\mu}\left(x, \theta_{t} \omega\right)=1, \text{ uniformly in }t \in \mathbb{R}.
$$

Next, using the fact
 $$
C\left(t+\tau ; \theta_{-\tau} \omega, \mu\right)=C\left(\tau ; \theta_{-\tau} \omega\right)+C(t ; \omega, \mu),$$
 we have that
$$
\begin{aligned}
&u(t, x\left.+C(t, \omega, \mu) ; U^{\mu}(\cdot, \omega), \omega\right) \\
&=\lim _{\tau \rightarrow \infty} u\left(t, x+C(t, \omega, \mu) ; u\left(\tau, x+C\left(\tau ; \theta_{-\tau} \omega, \mu\right) ; \phi_{\mu}^{+}, \theta_{-\tau} \omega\right), \omega\right) \\
&=\lim _{\tau \rightarrow \infty} u\left(t+\tau, x+C(t, \omega, \mu)+C\left(\tau ; \theta_{-\tau} \omega, \mu\right) ; \phi_{\mu}^{+}, \theta_{-\tau} \omega\right) \\
&=\lim _{\tau \rightarrow \infty} u\left(t+\tau, x+C\left(t+\tau, \theta_{-(t+\tau)} \theta_{t} \omega, \mu\right) ; \phi_{\mu}^{+}, \theta_{-(t+\tau)} \theta_{t} \omega\right) \\
&=U^{\mu}\left(x, \theta_{t} \omega\right).
\end{aligned}
$$
It follows from \eqref{3.10} and Proposition \ref{Pro4.2} that
$$
\lim _{x \rightarrow \infty} \sup _{t \in \mathbb{R}}\left|\frac{U^{\mu}\left(x, \theta_{t} \omega\right)}{\phi_{+}^{\mu}(x)}-1\right|=0.
$$
	Furthermore, since the function $\mathbb{R} \ni x \mapsto \phi_{+}^{\mu}$ is nonincreasing, then for every $\omega_{0} \in \Omega$ and every $t>0$, we have that the function $\mathbb{R} \ni x \mapsto u\left(t, x+C\left(t ; \theta_{-t} \omega,\mu\right) ; \phi_{+}^{\mu}, \theta_{-t}\omega\right)$ is decreasing, hence so is $U^{\mu}(\cdot, \omega)$. This completes the proof of Theorem \ref{Th2.3}(1).
\end{proof}
\begin{proof}[Proof of Theorem \ref{Th2.3}(2)]
Note that $u(t, x;\omega)=U^{\mu}\Big(x-C(t ; \omega, \mu), \theta_{t} \omega\Big)$ is a monotone random transition wave solution of \eqref{1.1}. We first claim that for a.e. $\omega\in\Omega$,
\begin{equation}\label{3.28d}
	C(t+s;\omega,\mu)=	C(t;\omega,\mu)+C(s;\theta_{t}\omega,\mu),\quad\forall s\geq0, t\in\R.
\end{equation}
Indeed, since $C(t;\omega,\mu)=\int_{0}^{t} c(s ; \omega, \mu)\ds$ and $c(t ; \omega, \mu)$ is defined in \eqref{2.3c}, we have that
\begin{equation}\label{3.29d}
	\begin{aligned}
		\int_{0}^{t+s}c(l;\omega,\mu)\dl&=\int_{0}^{t}c(l;\omega,\mu)\dl+\int_{t}^{t+s}c(l;\omega,\mu)\dl\\
		&=\int_{0}^{t}c(l;\omega,\mu)\dl+\int_{0}^{s}c(l+t;\omega,\mu)\dl.
	\end{aligned}
\end{equation}
Using \eqref{2.3c} and $$
\theta_{l+t}\omega=\theta_{l}\omega\theta_{t}\omega,\quad\forall l,t\in\R, \omega\in\Omega,
$$
we have that 
\begin{equation}\label{3.30d}
	\int_{0}^{s}c(l+t;\omega,\mu)\dl=\int_{0}^{s}c(l;\theta_{t}\omega,\mu)\dl.
\end{equation}
It follows from \eqref{3.29d} and \eqref{3.30d} that 
\begin{equation*}
		\int_{0}^{t+s}c(l;\omega,\mu)\dl=\int_{0}^{t}c(l;\omega,\mu)\dl+\int_{0}^{s}c(l;\theta_{t}\omega,\mu)\dl.
\end{equation*}
This implies that \eqref{3.28d} holds.
Then, applying \eqref{3.28d} and subadditive ergodic theorem, there is $c^*$ such that for a.e. $\omega\in\Omega$, there holds
\begin{equation*}
	\lim_{t\rightarrow\infty}\dfrac{C(t;\omega,\mu)}{t}=c^{*}.
\end{equation*}

Now, let $\mathbb{Q}$ be the set of rational numbers. Note that $U^{\mu}(\cdot, \omega)\in C_{unif}^{b}(\R\times\Omega;\R)$ is measurable in $\omega$ and $U^{\mu}(x, \omega)$ is bounded in $x$ and $\omega$. Using Lemma \ref{lem2.5d}, there is $\Omega_{1} \in \mathcal{F}$ with $\mathbb{P}\left(\Omega_{1}\right)=1$ such that $\lim _{t \rightarrow \infty} \frac{1}{t} \int_{0}^{t} U^{\mu}\left(x, \theta_{s} \omega\right) \mathrm{d} s$ exists for all $x \in \mathbb{Q}$ and $\omega \in \Omega_{1}$. Note that $U^{\mu}(x, \omega)$ is uniformly continuous in $x$. This implies that for all $x \in \mathbb{R}$ and $\omega \in \Omega_{1}$, the limit $\lim _{t \rightarrow \infty} \frac{1}{t} \int_{0}^{t} U^{\mu}\left(x, \theta_{s} \omega\right) \mathrm{d} s$ exists. Let
$$
U_{*}^{\mu}(x)=\lim _{t \rightarrow \infty} \frac{1}{t} \int_{0}^{t} U^{\mu}\left(x, \theta_{s} \omega\right) \mathrm{d}s
$$
for $x\in\R$ and $\omega\in\Omega_1$. We have $U_{*}^{\mu}\in C_{unif}^{b}(\R;\R)$ and $U_{*}^{\mu}$ is monotone.
\end{proof}
\begin{remark}
	Define $\mathscr{C} \subset L^{\infty}(\Omega,\mathcal{F},\mathbb{P})$ the set of admissible speed function, that is the set of the functions $C(t;\omega,\mu) \in L^{\infty}(\Omega,\mathcal{F},\mathbb{P})$ such that there exists a random transition front, according to Definition \ref{def-wave}, with the speed function $C(t;\omega,\mu)$ and $c(t;\omega,\mu)=C'(t;\omega,\mu)$. The above theorem ensures that
	$$
	\left\{t \mapsto c(t;\omega,\mu), \quad \mu \in\left(0, \mu^{*}\right) \right\} \subset \mathscr{C} .
	$$
	Therefore, recalling the definition of $c(t;\omega,\mu)$ in \eqref{2.3c} and Proposition \ref{Pro2.2} (iv), we obtain that
	$$
	\left(\lim _{\mu \rightarrow \mu^{*}}\lfloor c(t;\omega,\mu)\rfloor, \infty\right) \subset\lfloor\mathscr{C}\rfloor:=\{\lfloor c\rfloor, c \in \mathscr{C}\}.
	$$
\end{remark}

\section{ Stability of a random transition front }
In this subsection, we study the stability of random  transition fronts of \eqref{1.1}. Before we prove the main Theorem \ref{Th2.4}, we first give the following lemma.   
\begin{lemma}\label{Lem5.1}
Let $u_{0} \in C_{\text {unif }}^{b}(\mathbb{R})$ satisfy \eqref{2.4}. Then for any $\omega \in \Omega_{0}$, there holds
$$
	\lim _{x \rightarrow \infty} \frac{u\left(t, x+C(t ; \omega, \mu) ; u_{0}, \omega\right)}{e^{-\mu x}}=1 \quad \text { uniformly in } t \geq 0 \text {, }
	$$
	where $C(t ; \omega, \mu)=\int_{0}^{t} c(s ; \omega, \mu) d s\left(c(t,\omega,\mu):=\mu^{-1}\left[\int_{\mathbb{R}} J(\theta_{t}\omega, y)\left(\mathrm{e}^{\mu y}-1\right) \mathrm{d} y+a(\theta_{t}\omega)\right]\right)$ and $\mu$ is given by Theorem \ref{Th2.4}.
\end{lemma}
\begin{proof}
Since $u_{0}$ satisfies \eqref{2.4}, then for every $\varepsilon>0$, there is $x_{\varepsilon; \omega} \gg 1$ such that
$$
1-\varepsilon \leq \frac{u_{0}(x+C(0 ; \omega, \mu))}{U(x, \omega)} \leq 1+\varepsilon, \quad \forall x \geq x_{\varepsilon ; \omega}.
$$
Let $A_{\omega}(t)$ be as in Lemma \ref{Lem3.2}. Since $e^{-\mu x}-d_{\omega} e^{A_{\omega}(t)-\tilde{\mu} x} \leq U(x, t) \leq e^{-\mu x}$, then
\begin{equation}\label{5.1}
\begin{aligned}
(1-\varepsilon) e^{-\mu x}-(1-\varepsilon) d_{\omega} e^{A_{\omega}(0)-\tilde{\mu} x} & \leq u_{0}(x+C(0 ; \omega, \mu)) \leq(1+\varepsilon) e^{-\mu x}, \quad \forall x \geq x_{\varepsilon ; \omega}.
\end{aligned}
\end{equation}
We claim that there is $d\gg1$ such that
\begin{equation}\label{5.2}
\begin{aligned}
(1-\varepsilon) e^{-\mu x}-d e^{A_{\omega}(0)-\tilde{\mu}x}  \leq u_{0}(x+C(0 ; \omega, \mu))  \leq(1+\varepsilon) e^{-\mu x}+d e^{A_{\omega}(0) -\tilde{\mu}x}, \quad \forall x \in \mathbb{R} .
\end{aligned}
\end{equation}
Indeed, we observe that
$$
\begin{aligned}
u_{0}(x+C(0 ; \omega, \mu))\leq \left\|u_{0}\right\|_{\infty} e^{\tilde{\mu} x_{\varepsilon, \omega}} e^{-\tilde{\mu}x_{\varepsilon ; \omega}}\leq 
\left\|u_{0}\right\|_{\infty} e^{ \tilde{\mu}x_{\varepsilon ; \omega}+\left|A_{\omega}(0)\right|} e^{A_{\omega}(0)-\tilde{\mu}x}.
\end{aligned}
$$
Therefore,
\begin{equation}\label{5.3}
u_{0}(x+C(0 ; \omega, \mu)) \leq(1+\varepsilon) e^{-\mu x}+d_{\varepsilon, \omega} e^{A_{\omega}(0)-\tilde{\mu}x}, \quad \forall x \in \mathbb{R},
\end{equation}
where $d_{\varepsilon ; \omega}=:\left\|u_{0}\right\|_{\infty} e^{\tilde{\mu}x_{\varepsilon ; \omega}+\left|A_{\omega}(0)\right|}$.
On the other hand, for every $d>1$, the function $\mathbb{R} \ni x \mapsto(1-\varepsilon) e^{-\mu x}-d e^{A_{\omega}(0)-\tilde{\mu}x}$ reaches its maximum value at $x_{d}=\frac{\ln \left(\frac{d \tilde{\mu}\varepsilon_{\omega} A_{\omega}(0)}{(1-\varepsilon) \mu}\right)}{\tilde{\mu}-\mu}$.
It is clearly that
\begin{equation}\label{4.4}
	\lim _{d \rightarrow \infty} x_{d}=\infty,
\end{equation}
and
\begin{equation}\label{4.5}
	\lim _{d \rightarrow \infty}\left((1-\varepsilon) e^{-\mu x_{d}}-d e^{A_{\omega}(0)-\tilde{\mu}x_{d}}\right)=0.
\end{equation}
Therefore, 
	there is $\tilde{d}_{\varepsilon ; \omega} \gg(1-\varepsilon) d_{\omega}$ such that $x_{d_{\varepsilon ; \omega}} \geq x_{\varepsilon ; \omega}$ 
and
\begin{equation}\label{4.6}
	(1-\varepsilon) e^{-\mu x_{d_{\varepsilon} ; \omega}}-\tilde{d}_{\varepsilon ; \omega} e^{A_{\omega}(0)-\tilde{\mu}x_{\tilde{d}_{\varepsilon} ; \omega}} \leq \inf _{x \leq x_{\varepsilon; \omega}} u_{0}(x+C(0 ; \omega, \mu)).
\end{equation}
It follows from \eqref{5.1} and \eqref{4.6} that
\begin{equation}\label{5.4}
(1-\varepsilon) e^{-\mu x}-d e^{A_{\omega}(0)-\tilde{\mu}x} \leq u_{0}(x+C(0 ; \omega, \mu)), \quad \forall x \in \mathbb{R}, \quad \forall d \geq \tilde{d}_{\varepsilon ; \omega} .
\end{equation}
Therefore, by using \eqref{5.3} and \eqref{5.4}, we have that \eqref{5.2} holds for every $d \geq$ $\max \left\{\tilde{d}_{\varepsilon ; \omega}, d_{\varepsilon ; \omega}\right\}$. By direct computation as in the proof of Lemma \ref{Lem3.2}, it holds that for $d \gg 1$,
$$
	\mathcal{G}^{\omega, \mu}\left((1-\varepsilon) e^{-\mu x}-d e^{A_{\omega}(t)-\tilde{\mu}x}\right) \leq 0, \quad \text { a.e. in } t
	$$
on the set $D_{\varepsilon}:=\left\{(x, t) \in \mathbb{R} \times \mathbb{R}^{+} \mid(1-\varepsilon) e^{-\mu x}-d e^{A_{\omega}(t)-\tilde{\mu}x} \geq 0\right\}$. Therefore, by using 
$$
u\left(t, x+C(t ; \omega, \mu) ; u_{0}, \omega\right) \geq 0,
$$ 
and comparison principle, we have that
\begin{equation}\label{5.5}
\begin{array}{r}
(1-\varepsilon) e^{-\mu x}-d e^{A_{\omega}(t)-\tilde{\mu}x} \leq u\left(t, x+C(t ; \omega, \mu) ; u_{0}, \omega\right),\quad
\forall x \in \mathbb{R}, \quad \forall t \geq 0, d \gg 1.
\end{array}
\end{equation}
Similarly, we also have that
$$
\mathcal{G}^{\omega, \mu}\left((1+\varepsilon) e^{-\mu x}+d e^{A_{\omega}(t)-\tilde{\mu}x}\right) \geq 0, \quad x \in \mathbb{R}, \quad t \in \mathbb{R}.
$$
Then, it follows from comparison principle that
\begin{equation}\label{4.9}
	u\left(t, x+C(t ; \omega, \mu) ; u_{0}, \omega\right) \leq(1+\varepsilon) e^{-\mu x}+d e^{A_{\omega}(t)-\tilde{\mu}x}, \quad \forall x \in \mathbb{R}, \forall t \geq 0, d \gg 1 .
\end{equation}
By using \eqref{5.5}, \eqref{4.9} and arbitrariness of $\varepsilon>0$, we have that
	$$
	\lim _{x \rightarrow \infty} \frac{u\left(t, x+C(t ; \omega, \mu) ; u_{0}, \omega\right)}{e^{-\mu x}}=1 \quad \text { uniformly in } t \geq 0 \text {. }
	$$
Therefore, we complete the proof of Lemma \ref{Lem5.1}.
\end{proof}

Now we are ready to prove the Theorem \ref{Th2.4}.
\begin{proof}[Proof of Theorem \ref{Th2.4}]
Fix $\omega \in \Omega_{0}$. Let $u_{0} \in C_{\text {unif }}^{b}(\mathbb{R})$ satisfying \eqref{2.4}. Then there exists $\alpha \geq 1$ such that
\begin{equation}\label{6.10}
	\frac{1}{\alpha} \leq \frac{u_{0}(x)}{U(x-C(0 ; \omega, \mu), \omega)} \leq \alpha, \quad \forall x \in \mathbb{R}.
\end{equation}
Therefore, by comparison principle, we have that
\begin{equation*}
	u\left(t, x ; u_{0}, \omega\right) \leq u(t, x ; \alpha U(\cdot-C(0 ; \omega, \mu), \omega)), \quad \forall x \in \mathbb{R}, \quad \forall t \geq 0,
\end{equation*}
and
$$
U\left(x-C(t ; \omega, \mu), \theta_{t} \omega\right) \leq u\left(t, x ; \alpha u_{0}, \omega\right), \quad \forall x \in \mathbb{R}, \quad \forall t \geq 0.
$$
It is clearly that
\begin{equation}\label{6.11}
	(\alpha u)_{t} \geq\alpha \int_{\R}J(\theta_{t}\omega, y)[u(t,x-y)-u(t,x)]\dy+a\left(\theta_{t} \omega\right)(\alpha u)(1-\alpha u).
\end{equation}
Then by comparison principle, we have that
\begin{equation}\label{6.12}
	U\left(x-C(t ; \omega, \mu), \theta_{t} \omega\right) \leq \alpha u\left(t, x ; u_{0}, \omega\right), \quad \forall t \geq 0.
\end{equation}
Similarly, we have that
\begin{equation}\label{6.13}
	u\left(t, x ; u_{0}, \omega\right) \leq \alpha U\left(x-C(t ; \omega, \mu), \theta_{t} \omega\right), \quad \forall t \geq 0.
\end{equation}
Therefore, $\forall t \geq 0$, there exists a unique $\alpha(t) \geq 1$ satisfying
\begin{equation}\label{5.6}
\alpha(t):=\inf \left\{\alpha \geq 1 \mid \frac{1}{\alpha} \leq \frac{u\left(t, x ; u_{0},\omega\right)}{U\left(x-C(t ; \omega, \mu), \theta_{t} \omega\right)} \leq \alpha, \forall x \in \mathbb{R}\right\}.
\end{equation}
Moreover, we have that 
\begin{equation}\label{6.15a}
	\alpha(t) \leq \alpha(\tau), \quad\forall 0 \leq \tau \leq t.
\end{equation}
Indeed, by using \eqref{6.10}, we define
\begin{equation*}
	\alpha(0)=\inf \left\{\alpha \geq 1 \mid \frac{1}{\alpha} \leq \frac{u_{0}(x)}{U(x-C(0 ; \omega, \mu), \omega)} \leq \alpha, \quad \forall x \in \mathbb{R}\right\}.
\end{equation*}
It follows from \eqref{6.12} and \eqref{6.13} that
\begin{equation}\label{6.15}
	\frac{1}{\alpha(0)} \leq \frac{u\left(t, x ; u_{0},\omega\right)}{U\left(x-C(t ; \omega, \mu), \theta_{t} \omega\right)} \leq \alpha(0), \forall t>0, x \in \mathbb{R}.
\end{equation}
This inequality is similar to \eqref{6.10}, thus $\alpha(t)$ given by \eqref{5.6} is well defined for all $t>0$. Moreover, since \eqref{6.15} holds $\forall t>0$, then 
\begin{equation}\label{6.16}
	\alpha(t)\leq \alpha(0), \forall t>0.
\end{equation}
We now fix $0<\tau<t$. Using the definition of $\alpha(\tau)$, we have that
\begin{equation}\label{6.17}
	\frac{1}{\alpha(\tau)} \leq \frac{u\left(\tau, x ; u_{0},\omega\right)}{U\left(x-C(\tau ; \omega, \mu), \theta_{\tau} \omega\right)} \leq \alpha(\tau), \forall x \in \mathbb{R}.
\end{equation}
Let $\tau$ be the initial time, we replace $u_0(x)$ and $U(x-C(0;\omega,\mu),\omega)$ by $u\left(\tau, x ; u_{0},\omega\right)$ and $U\left(x-C(\tau ; \omega, \mu), \theta_{\tau} \omega\right)$ in the arguments \eqref{6.10}--\eqref{6.13}, respectively. 
Similar to the derivation of \eqref{6.16}, we have 
\begin{equation}\label{6.18}
	\alpha(t)\leq\alpha(\tau), \forall t>\tau.
\end{equation}
Using \eqref{6.16} and \eqref{6.18}, we have that \eqref{6.15a} holds.
Therefore, we can define $\alpha_{\infty}$ as follows
$$
\alpha_{\infty}:=\inf \{\alpha(t) \mid t \geq 0\}=\lim _{t \rightarrow \infty} \alpha(t).
$$
To complete the proof of Theorem \ref{Th2.4}, it is only to show that $\alpha_{\infty}=1$.
It is clear that 
$$\alpha_{\infty} \geq 1.$$ We assume by contradiction that $\alpha_{\infty}>1$. Let $1<\alpha<\alpha_{\infty}$ be fixed. It follows from Lemma \ref{Lem5.1} that there exists $x_{\alpha} \gg 1$ such that
	\begin{equation}\label{5.7}
		\frac{1}{\alpha} \leq \frac{u\left(t, x+C(t ; \omega, \mu) ; u_{0}, \omega\right)}{U\left(x, \theta_{t} \omega\right)} \leq \alpha, \quad \forall x \geq x_{\alpha}, \forall t \geq 0 .
	\end{equation}
Set
$$
m_{\alpha}:=\frac{1}{\alpha_{0}} \inf _{t \geq 0, x \leq x_{\alpha}} U\left(x ; \theta_{t} \omega\right)>0,
$$
where $\alpha_{0}=\alpha(0)=\sup _{t \geq 0} \alpha(t)$. Hence it follows from the definition of $\alpha_{0}$ that
	$$
	m_{\alpha} \leq \min \left\{u\left(t, x+C(t ; \omega, \mu) ; u_{0}, \omega\right), U\left(x, \theta_{t} \omega\right)\right\}, \quad \forall x \leq x_{\alpha}, \forall t \geq 0 .
	$$
By using Assumption (H2), we have that there exists $T=T(\omega) \geq 1$ such that
	\begin{equation}\label{5.8}
		0<\frac{\lfloor a \rfloor T}{2}<\int_{s}^{s+T} a\left(\theta_{\tau} \omega\right) d s<2 \lceil a\rceil T<\infty, \quad \forall s \in \mathbb{R} .
	\end{equation}
Let $0<\delta \ll 1$ satisfy
\begin{equation}\label{5.9}
\alpha<e^{-2 \delta T \lceil a\rceil} \alpha_{\infty} \text { and }\left(\left(\alpha_{\infty}-1\right)-\alpha_{0}\left(1-e^{-2 \delta T\lceil a\rceil}\right)\right) m_{\alpha}>\delta .
\end{equation}
We claim that
\begin{equation}\label{5.10a}
\alpha((k+1) T) \leq e^{-\delta \int_{k T}^{(k+1) T} a\left(\theta_{s} \omega\right) d s} \alpha(k T), \quad \forall k \geq 0.
\end{equation}
Indeed, we take
\begin{equation}\label{4.15}
	u_{k}(t, x)=e^{\delta \int_{k T}^{t+k T} a\left(\theta_{s} \omega\right) d s} u\left(t+k T, x+C(t+k T ; \omega, \mu) ; u_{0}, \omega\right),\quad U_{k}(t, x)=U\left(x ; \theta_{t+k T} \omega\right),
\end{equation}
and
\begin{equation}\label{4.16}
	a_{k}(t)=a\left(\theta_{t+k T} \omega\right),\quad \alpha_{k}=\alpha(k T),\quad J_{k}(t,y)=J(\theta_{t+kT}\omega, y).
\end{equation} 
On the one hand, it follows from \eqref{5.8}, \eqref{4.15} and \eqref{4.16} that for any $t \in(0, T), x \in \mathbb{R}$, and $k \geq 0$
\begin{equation}\label{5.11}
\begin{aligned}
\partial_{t} u_{k}&=\delta a_{k}(t) u_{k}+\int_{\R}J_{k}(t,y)[u_k(t,x-y)-u_k(t,x)]\dy+\frac{\int_{\mathbb{R}} J_{k}(t,y)\left(\mathrm{e}^{\mu y}-1\right) \mathrm{d} y+a_k(t)}{\mu}\partial_{x} u_{k}\\
&\quad+a_{k}(t)\left(1-u\left(t+k T, x+C(t+k T ; \omega, \mu) ; u_{0}, \omega\right)\right) u_{k}\\
&=\int_{\R}J_{k}(t,y)[u_k(t,x-y)-u_k(t,x)]\dy+\frac{\int_{\mathbb{R}} J_{k}(t,y)\left(\mathrm{e}^{\mu y}-1\right) \mathrm{d} y+a_k(t)}{\mu}\partial_{x} u_{k}\\
&\quad+a_{k}(t)\left(1-u_{k}\right) u_{k}+a_{k}(t)\left(\left(1-e^{-\delta \int_{k}^{t+k T} a\left(\theta_{s} \omega\right) d s}\right) u_{k}+\delta\right) u_{k}\\
&\leq\int_{\R}J_{k}(t,y)[u_k(t,x-y)-u_k(t,x)]\dy+\frac{\int_{\mathbb{R}} J_{k}(t,y)\left(\mathrm{e}^{\mu y}-1\right) \mathrm{d} y+a_k(t)}{\mu}\partial_{x} u_{k}\\
&\quad+a_{k}(t)\left(1-u_{k}\right) u_{k}
+a_{k}(t)\left(\left(1-e^{-2 \delta t \lceil a\rceil}\right) u_{k}+\delta\right) u_{k}.
\end{aligned}
\end{equation}
On the other hand, it follows from \eqref{5.9} and the fact $\alpha_{\infty} \leq \alpha_{k} \leq \alpha_{0}$ that for 
$x \leq x_{\alpha}, 0 \leq t \leq T$ and $k \geq 0$
\begin{equation}\label{5.12}
\begin{aligned}
&\partial_{t}\left(\alpha_{k} U_{k}\right)-\alpha_{k}\int_{\R}J_{k}(t,y)[U_k(t,x-y)-U_k(t,x)]\dy-\frac{\int_{\mathbb{R}} J_{k}(t,y)\left(\mathrm{e}^{\mu y}-1\right) \mathrm{d} y+a_k(t)}{\mu}\partial_{x}\left(\alpha_{k} U_{k}\right) \\
=& a_{k}(t)\left(1-U_{k}\right)\left(\alpha_{k} U_{k}\right) \\
=& a_{k}(t)\left(1-\left(\alpha_{k} U_{k}\right)\right)\left(\alpha_{k} U_{k}\right)+a_{k}(t)\left(\left(1-e^{-2 \delta T \lceil a\rceil}\right)\left(\alpha_{k} U_{k}\right)+\delta\right)\left(\alpha_{k} U_{k}\right) \\
&+a_{k}(t)\left(\left(\left(\alpha_{k}-1\right)-\left(1-e^{-2 \delta T \lceil a\rceil}\right) \alpha_{k}\right) U_{k}-\delta\right)\left(\alpha_{k} U_{k}\right) \\
\geq & a_{k}(t)\left(1-\left(\alpha_{k} U_{k}\right)\right)\left(\alpha_{k} U_{k}\right)+a_{k}(t)\left(\left(1-e^{-2 \delta T \lceil a\rceil}\right)\left(\alpha_{k} U_{k}\right)+\delta\right)\left(\alpha_{k} U_{k}\right) \\
&+a_{k}(t)\left(\left(\left(\alpha_{\infty}-1\right)-\left(1-e^{-2 \delta T \lceil a\rceil}\right) \alpha_{0}\right) m_{\alpha}-\delta\right)\left(\alpha_{k} U_{k}\right) \\
\geq & a_{k}(t)\left(1-\left(\alpha_{k} U_{k}\right)\right)\left(\alpha_{k} U_{k}\right)+a_{k}(t)\left(\left(1-e^{-2 \delta T \lceil a\rceil}\right)\left(\alpha_{k} U_{k}\right)+\delta\right)\left(\alpha_{k} U_{k}\right).
\end{aligned}
\end{equation}
Using \eqref{5.9}, we have that
\begin{equation}\label{4.19}
	e^{\delta \int_{k T}^{(k+1) T} a\left(\theta_{s} \omega\right) d s} \alpha \leq \alpha_{\infty} \leq \alpha_{k}.
\end{equation}
Therefore, it follows from the definition of $\alpha_k$, \eqref{5.7}, \eqref{4.19} and comparison principle for equations that
	$$
	\begin{array}{r}
		e^{\delta \int_{k T}^{t+k T} a\left(\theta_{s} \omega\right) d s} u\left(t+k T, x+C\left(t+k T ; u_{0}, \omega\right) \leq \alpha_{k} U\left(x, \theta_{t+k T} \omega\right)\right., \\
		\forall x \leq x_{\alpha}, \quad t \in[0, T], k \geq 0.
	\end{array}
$$
That is
$$
\begin{array}{r}
	u\left(t+k T, x+C(t+k T ; \omega) ; u_{0}\right) \leq e^{-\delta \int_{k T}^{(k+t) T} a\left(\theta_{s} \omega\right) d s} \alpha_{k} U\left(x, \theta_{t+k T} \omega\right), \\
	\forall x \leq x_{\alpha}, \quad t \in[0, T], k \geq 0.
\end{array}
$$
By \eqref{4.19}, we have that
\begin{equation}\label{4.20}
\alpha \leq e^{-\delta \int_{k T}^{(k+1) T} a\left(\theta_{s} \omega\right) d s} \alpha_{\infty} \leq e^{-\delta \int_{k T}^{(k+1) T} a\left(\theta_{s} \omega\right) d s} \alpha_{k}.
\end{equation}
Therefore, it follows from \eqref{5.7} and \eqref{4.20} that
$$
\begin{array}{r}
	u\left(t+k T, x+C(t+k T ; \omega) ; u_{0}, \omega\right) \leq e^{-\delta \int_{k T}^{t+k T} a\left(\theta_{s} \omega\right) d s} \alpha_{k} U\left(x, \theta_{t+kT} \omega\right), \\
	\forall x \geq x_{\alpha}, \quad t \in[0, T], k \geq 0.
\end{array}
$$
Therefore, for every $k \geq 1$, it holds that
\begin{equation}\label{5.13}
	u\left(t+k T, x+C\left(t+k T ; u_{0}, \omega\right) ; u_{0}, \omega\right) \leq e^{-\delta \int_{k T}^{t+k T} a\left(\theta_{s} \omega\right) d s} \alpha_{k} U\left(x, \theta_{t+k T} \omega\right),\forall x \in \mathbb{R},  t \in[0, T].
\end{equation}
Similarly, interchanging $u_{k}$ and $U_{k}$ in \eqref{5.11} and \eqref{5.12}, we obtain that
\begin{equation}\label{5.14}
	\begin{gathered}
		U\left(x, \theta_{t+k T} \omega\right) \leq e^{-\delta \int_{k T}^{t+k T} a\left(\theta_{s} \omega\right) d s} \alpha_{k} u\left(t+k T, x+C\left(t+k T ; u_{0}, \omega\right) ; u_{0}, \omega\right), 
		\forall x \in \mathbb{R}, t \in[0, T] .
	\end{gathered}
\end{equation}
Hence inequality \eqref{5.10a} follows from \eqref{5.6}, \eqref{5.13} and \eqref{5.14}. Thus, by induction we obtain that
\begin{equation}\label{4.23}
\alpha_{\infty} \leq \alpha((k+1) T) \leq e^{-\delta \sum_{i=0}^{k} \int_{i T}^{(i+1) T} a\left(\theta_{s} \omega\right) d s} \alpha(0)=e^{-\delta \int_{0}^{(k+1) T} a\left(\theta_{s} \omega\right) d s} \alpha_{0}, \quad \forall k \geq 0.
\end{equation}
But for $\omega \in \Omega_{0}$, it holds that $\int_{0}^{\infty} a\left(\theta_{s} \omega\right) d s=\infty$. Therefore, letting $k \rightarrow \infty$ in \eqref{4.23}, we obtain that 
$$
\alpha_{\infty} \leq 0,
$$
which is impossible because $\alpha_{\infty} \geq 1$. Therefore $\alpha_{\infty}=1$, which completes the proof of the theorem.
\end{proof}

\section*{Declarations}
\section*{Ethical Approval}
This declaration is ``not applicable''. 
\section*{Competing interests}
The authors declare that they have no known competing financial interests or personal relationships that could have appeared to influence the work reported in this paper.

\section*{Authors' contributions}
Min Zhao contributed to the writing of the manuscript, and Rong Yuan made detailed revisions to the manuscript. Both authors discussed the results and contributed to the final manuscript.

\section*{Funding}
This research is supported by the China
Scholarship Council (CSC) to Min Zhao (202006040239). 
The authors of this research are also grateful for the help of the University of Bordeaux. This work is also supported by the National Natural Science Foundation of China (No. 11871007 and 12171039).

\section*{Availability of data and materials}
The authors declare that no data and materials were used in this paper.

\section*{Conflict of interest statement}
The authors declare that they have no known competing financial interests or personal relationships that could have appeared to influence the work reported in this paper.

\end{document}